\documentclass[preprint,10pt,authoryear]{elsarticle}

\usepackage{amssymb}
\usepackage{natbib}
\usepackage{graphicx}
\usepackage{color}
\usepackage{amsmath}
\usepackage{amssymb}
\usepackage{amscd}
\usepackage{bbm}
\usepackage{tikz}
\usepackage{hyperref}

\usepackage[utf8]{inputenc}
\usepackage{amsthm}
\usepackage{bbm}
\usepackage{algorithm}
\usepackage{algorithmic}
\usepackage{hyperref}
\hypersetup{
    bookmarks=true,         % show bookmarks bar?
    colorlinks=true,       % false: boxed links; true: colored links
    linkcolor=red,          % color of internal links (change box color with linkbordercolor)
    citecolor=green,        % color of links to bibliography
    filecolor=magenta,      % color of file links
    urlcolor=cyan           % color of external links
}

\newtheorem{Theorem}{Theorem}
\newtheorem{Assumption}{Assumption}
\newtheorem{Definition}{Definition}
\newtheorem{Lemma}{Lemma}
\newtheorem{Proposition}{Proposition}

\newtheorem{Remark}{Remark}

\newcommand{\norm}[1]{\left\|#1\right\|}%
\newcommand{\an}{{\mbox{ and }}}

\newcommand{\inr}[1]{\bigl< #1 \bigr>}

\newcommand{\pa}[1]{\left(#1\right)}
\newcommand{\cro}[1]{\left\{#1\right\}}

\newcommand\eps{\epsilon}

\newcommand{\est}{{\hat f}_{K,\lambda}}%
\newcommand{\esti}{{\hat f}_{K}}

\def \endproof
{{\mbox{}\nolinebreak\hfill\rule{2mm}{2mm}\par\medbreak}}

\DeclareMathOperator*{\argmin}{argmin}

\def\ds1{\textrm{1\kern-0.25emI}} %{\mathds{1}}

\newcommand \E{\mathbb{E}}
\newcommand \R{\mathbb{R}}
\newcommand \N{\mathbb{N}}
\newcommand \cB{{\cal B}}
\newcommand \cC{{\cal C}}

\newcommand \cF{{\cal F}}
\newcommand \cG{{\cal G}}

\newcommand \cI{{\cal I}}
\newcommand \cJ{{\cal J}}
\newcommand \cK{{\cal K}}

\newcommand \cO{{\cal O}}

\newcommand \cX{{\cal X}}

\newcommand \bE{{\mathbb E}}

\newcommand \bP{{\mathbb P}}

\newcommand \fM{{\mathfrak M}}

\newcommand \fQ{{\mathfrak Q}}

\newcommand{\hf}{ {\hat f}}

\newcommand{\bayes}{f^*}
\newcommand{\ESTI}[2]{\widehat f_{#1}^{(#2)}}
\newcommand{\MOM}[2]{\text{MOM}_{#1}\left[#2\right]}

\newcommand{\cabs}[1]{c_{#1}}
\newcommand{\param}[1]{\theta_{#1}}

\begin{document}

\begin{frontmatter}

%% Title, authors and addresses

%% use the tnoteref command within \title for footnotes;
%% use the tnotetext command for theassociated footnote;
%% use the fnref command within \author or \address for footnotes;
%% use the fntext command for theassociated footnote;
%% use the corref command within \author for corresponding author footnotes;
%% use the cortext command for theassociated footnote;
%% use the ead command for the email address,
%% and the form \ead[url] for the home page:
%% \title{Title\tnoteref{label1}}
%% \tnotetext[label1]{}
%% \author{Name\corref{cor1}\fnref{label2}}
%% \ead{email address}
%% \ead[url]{home page}
%% \fntext[label2]{}
%% \cortext[cor1]{}
%% \address{Address\fnref{label3}}
%% \fntext[label3]{}

\title{Learning from MOM's principles : Le Cam's approach}

%% use optional labels to link authors explicitly to addresses:
%% \author[label1,label2]{}
%% \address[label1]{}
%% \address[label2]{}

\author[label1]{Guillaume Lecu\'e}
\ead[label1]{guillaume.lecue@ensae.fr}
\address[label1]{CREST, CNRS, Universit\'e Paris Saclay}

\author[label2]{Matthieu Lerasle}
\ead[label2]{matthieu.lerasle@math.u-psud.fr}
\address[label2]{Laboratoire de Mathématiques d'Orsay, Univ. Paris-Sud, CNRS, Universit\'e Paris Saclay}

\begin{abstract}
We obtain estimation error rates for estimators obtained by aggregation of regularized median-of-means tests, following a construction of Le Cam. The results hold with exponentially large probability,
% -- as in the gaussian framework with independent noise-- 
 under only weak moments assumptions on data.
%  and without assuming independence between noise and design.
Any norm may be used for regularization. When it has some sparsity inducing power we recover sparse rates of convergence. 
The procedure is robust since a large part of data may be corrupted, these outliers have nothing to do with the oracle we want to reconstruct. Our general risk bound is of order
\begin{equation*}
\max\left(\mbox{minimax rate in the i.i.d. setup}, \frac{\mbox{number of outliers}}{\mbox{number of observations}}\right)  \enspace.
\end{equation*}In particular, the number of outliers  may be as large as \textit{(number of data) $\times$(minimax rate)} without affecting this rate. The other data do not have to be identically distributed but should only have equivalent $L^1$ and $L^2$ moments.
For example, the minimax rate $s \log(ed/s)/N$ of recovery of a $s$-sparse vector in $\R^d$ is achieved with exponentially large probability by a median-of-means version of the LASSO when the noise has $q_0$ moments for some $q_0>2$, the entries of the design matrix should have  $C_0\log(ed)$ moments and the dataset can be corrupted up to $C_1 s \log(ed/s)$ outliers.
% and the result holds . 
%This result holds with exponentially large probability as if the noise and the design were i.i.d. Gaussian random variables. 
\end{abstract}

\begin{keyword}
robust statistics \sep statistical learning \sep high dimensional statistics.
%% keywords here, in the form: keyword \sep keyword

%% PACS codes here, in the form: \PACS code \sep code

%% MSC codes here, in the form: 
\MSC[2010]  62G35  \sep 62G08.
%% or \MSC[2008] code \sep code (2000 is the default)

\end{keyword}

\end{frontmatter}

\section{Introduction}
Consider the problem of estimating minimizers of the integrated square-loss over a convex class of functions : $\bayes\in\argmin_{f\in F}P(Y-f(X))^2$ based on a data set $(X_i,Y_i)_{i=1,\ldots, N}$. The labels $Y$ and $Y_i$'s are real-valued while the inputs $X$ and $X_i$'s take values in an abstract measurable space $\cX$.

 Empirical Risk Minimizers (ERM) of \cite{MR1641250, MR0474638} and later on, their regularized versions replace the unknown distribution $P$ in the definition of $\bayes$ by the empirical distribution $P_N$ based on the sample $(X_i,Y_i)_{i=1,\ldots, N}$. Given a function ${\rm reg}: F\to \R_+$, this produces regularized ERM defined by 
\[
\hat f^{\text{RERM}}_N\in\argmin_{f\in F}\{P_N(Y-f(X))^2+{\rm reg}(f)\}\enspace.
\]
These estimators are optimal in i.i.d. subgaussian setups but suffer several drawbacks when data are heavy-tailed or corrupted by ``outliers", see \cite{MR3052407,HubRonch2009}. These issues are critical in many modern applications such as high-frequency trading, where heavy-tailed data are quite common or in various areas of biology such as micro-array analysis or neuroscience where data are sometimes still nasty after being preprocessed. To overcome the problem, various methods have been proposed. The most common strategy is to replace the square-loss function to make it less sensitive to outliers. For example, \cite{MR0161415} proposed a loss that
interpolates between square and absolute loss to produce an estimator between the unbiased (but non robust) empirical mean and the (more robust but biased) empirical median. Huber's estimators have been intensively studied asymptotically by \cite{MR0161415,HubRonch2009}, non-asymptotic results have also been obtained more recently by \cite{MR3217454,shahar_general_loss, FanLiWang2016} for example. An alternative approach has been proposed by \cite{MR3052407} and used in learning frameworks such as least-squares regression by \cite{MR2906886} and for more general loss functions by \cite{MR3405602}.

Another line of research to build robust estimators and robust selection procedures was initiated by \cite{MR0334381, MR856411} and further developed by \cite{MR2219712}, \cite{MR2834722} and \cite{BaraudBirgeSart}. It is based on \emph{comparisons} or \emph{tests} between elements of $F$. More precisely, the approach builds on tests statistics $T_N(g,f)$ comparing $f$ and $g$. These tests define the sets $\cB_{T_N}(f)$ of all $g$'s that have been preferred to $f$ and the final estimator $\hat f$ is a minimizer of the diameter of $\cB_{T_N}(f)$. The measure of diameter is directly related to statistical performances one seeks for the estimator. These methods mostly focus on Hellinger loss and are generally considered difficult to compute, see however \cite{MR3224300, MR3224298}.

In a related but different approach, \cite{LugosiMendelson2016} have recently introduced ``median-of-means tournaments". Median-of-means estimators of \cite{MR1688610, MR855970, MR702836} compare elements of $F$. A ``champion" is an element $\hat f$ such that $\cB_{T_N}(\hat f)$ is smaller than a computable upper bound on the radius of $\cB_{T_N}(\bayes)$. They prove that the risk of any champion is controlled by this upper bound. An important message of this paper is that Le Cam's estimators are quite common in statistics, in particular in robust statistics. For example, Section~\ref{Sec:LFT} shows that any penalized empirical loss function can be obtained by Le Cam's approach and that Le Cam's estimators based on median-of-means tests are champions of median-of-means tournaments.

This paper studies estimators derived from Le Cam's procedure based on regularized median-of-means (MOM) tests (see Section~\ref{Sec:QOMProcesses}). Our estimators are therefore particular instances of champions of MOM's tournaments and another motivation is to push further the analysis of this particular champion. The main advantage of MOM's tests over Le Cam's original ones is that they allow for more classical loss functions than Hellinger loss. This idea is illustrated on the square-loss. Compared to Huber or Catoni's losses, this approach allows to control easily the risk of our estimators by using classical tools from empirical process theory, it also allows to tackle the problem of ``aggressive" outliers.

The closest work is certainly that of \cite{LugosiMendelson2016}, but we believe that our paper contains substantial improvements. We stress the intimate relationship between their estimator and Le Cam general construction and use this parallel to propose a much simpler estimator. Our risk bounds are always better and we extend their results to possibly corrupted data-sets. 
%We also select the number of blocks used to construct MOM's tests in a data-dependent way. 

To investigate robustness properties of median-of-means estimators, we partition the dataset into two parts. One is made of outliers data. They are indexed by $\cO\subset[N]$ of cardinality $|\cO|=K_o$. \textbf{On those data, absolutely nothing is assumed }: they may not be independent, have distributions $P_i$ totally different from $P$, with no moment at all, etc.. These are typically data polluting datasets like in the case of declarative data on internet or when something went wrong during the storage, compression or transfer which resulted in complete non sense data. They may also be observations met in biology as in the classical \textit{eQTL (Expression Quantitative Trait Loci and The Phenogen Database)} from \cite{eQTL}. Many other examples of datasets containing outliers could be provided, this includes frauds detection and terrorist activity as examples. Of course, outliers are not flagged in advance and the statistician is given no a priori information on which data is an outlier or not.
The other part of the dataset is made of data on which the MOM estimator rely on to estimate the oracle $f^*$. There should be enough information in those data so that the estimation of $f^*$ is possible, even in the presence of outliers provided they remain in a ``decent proportion''. We therefore call the non-outliers, the \textit{informative data}, those that bring information on $f^*$. We denote by $\cI\subset[N]$ the set indexing these data. We therefore end up with a partition of $[N]$ as $[N] = \cI \cup \cO$ which, again, is not known from the statistician.

%More precisely, we build two different types of estimators and prove three different types of results. First, t
The radii of the sets $\cB_{T_N}(f)$ are computed for regularization and $L^2_P$ norms. The regularization norm is chosen in advance by the statistician to  promote sparsity or smoothness. It can be used freely in our procedure, but it doesn't ensure a small $L^2_P$ risk for the estimator. The $L^2_P$-norm is unknown in general since it depends on the distribution of $X$. Furthermore, the classical $L^2_{P_N}$-empirical metric fails to estimate the $L^2_P$ metric without subgaussian properties of the design vector $X$. Fortunately, it can be replaced by a median-of-means metric. To handle simultaneously both regularization and $L^2_P$ norms, we will also slightly extend Le Cam's principle. 
Our first important result shows that the resulting estimator is well localized w.r.t. both regularization and $L^2_P$ norms. 
%Note that the third criteria we added to Birg\'e's procedure is based on the median of means principle to control the correlation between the noise $Y-\bayes(X)$ and $(f-\bayes)(X)$ for all $f\in F$. 
% The resulting estimator, besides being well localized and having a nice excess risk, satisfies furthermore an ``oracle inequality", meaning that the least-squares loss of the estimator is bounded by the one of the oracle plus a bound proportional to the optimal excess risk. 

Median-of-means estimators rely on a data splitting into $K$ blocks and this parameter drives the resulting statistical performances (cf. \cite{MR3576558}). To achieve optimal rates, $K$ should be ultimately chosen using parameters that depend on the oracle $\bayes$ like its sparsity which is not in general available to the statistician. To bypass this problem, the strategy of \cite{MR1147167} is used as in \cite{MR3576558} to select $K$ adaptively and get a fully data-driven procedure. 

There are four important features in our approach. First, all results are proved under weak $L^{2+\epsilon}$ moment assumptions on the noise. This is an almost minimal condition for the problem to make sense. The class $F$ is only assumed to satisfy a weak ``$L_2/L_1$'' comparison. Second, performances of the estimators are not affected by the presence of complete outliers, as long as their number remains comparable to \textit{(number of observations)$\times$(rates of convergence)}.
% which is, for the problem of sparse-recovery, the sparsity of the oracle. 
Third, all results are non-asymptotic and the regression function $x\mapsto\E[Y|X=x]$ is never assumed to belong to the class $F$. In particular, the noise $Y-\bayes(X)$ can be correlated with $X$. Finally, even ``informative data", those that are not ``outliers", are not requested to be i.i.d. $\sim P$, but only to have close first and second moments for all $f\in F-\{f^*\}$. Nevertheless, the estimators are shown to behave as well as the ERM when the data are i.i.d. $\sim P$, $\E[Y|X=\cdot]\in F$, the noise $\zeta=Y-\bayes(X)$ and the class $F$ are Gaussian and the noise is independent from the design.

\textbf{Example: sparse-recovery via MOM LASSO.}
As a proof of concept, theoretical properties are illustrated in the classical example of sparse-recovery in high-dimensional spaces using the $\ell_1$-regularization. 
This example illustrates typical results that follow from our analysis in one of the most classical problem of high dimensional statistics (cf. \cite{MR2807761,MR3307991}). 
The interested reader can check that it also applies to other procedures like Slope (cf. \cite{slope1,slope2}) and trace-norm regularization as well as kernel methods, for instance, by using the results in \cite{LM_reg_comp,LM_reg_comp_2}. 
%
%This example shows the minimax optimality of our rates in the hostile environment with outliers and where informative data are only assumed to have few finite moments. In addition, the parameter $\lambda$ used to balance between adequacy and regularization is chosen adaptively in our final estimator, it does not depend on the sparsity parameter as it is the case for the Lasso estimator. Let us now show the advantage of our procedure over the classical regularization methods for the problem of sparse recovery. We f

Recall this classical setup.
%
%, then recall a result from \cite{LM_reg_comp} on the LASSO and then state our result for this example.
%
Let $X$ denote a random vector in $\R^d$ such that $\E\inr{X, t}^2=\norm{t}_2^2$ for all $t\in\R^d$ ($X$ is isotropic) and let $Y$ be a real-valued random vector. Let $t^*\in\argmin_{t\in\R^d}\E(Y-\inr{X, t})^2$. 
Let $(X_i,Y_i)_{i\in[N]}$ denote independent data corrupted by outliers : no assumption is made on a subset $(X_i,Y_i)_{i\in \cO}$ of the dataset. Let $\cI=[N]\setminus \cO$ denote the indices of \textit{informative data} $(X_i,Y_i)_{i\in \cI}$: for all $i\in \cI$, $(X_i, Y_i)$ are independent with the same distribution $(X, Y)$.
% -- the results also apply when $(X_i, Y_i)_{i\in\cI}$ are not identically distributed but with first and second moment equivalent to the one of $P$, but f
 For the sake of simplicity, we only consider the case of i.i.d. informative data in this example.
 In high-dimensional statistics, $N\leq d$ but $t^*$ has only $s$ ($s < N$) non-zero coordinates. To estimate $t^*$, the $\ell_1$-norm $\norm{\cdot}_1$ is used for penalization to promote zero coordinates. The following result holds.

\begin{Theorem}\label{theo:lasso_classi}[Theorem~1.4 in \cite{LM_reg_comp}] Assume $t^*$ is $s$-sparse, $N\geq c_0 s \log(ed/s)$, $X$ is isotropic and
\begin{enumerate}
	\item[i)] $|\cI|=N$ and $|\cO|=0$ (no outliers in the dataset),
	\item[ii)] $ \zeta=Y-\inr{X, t^*} \in L_{q_0}$ for some $q_0>2$
	\item[iii)] there exists $L>0$ such that for all $t\in\R^d$ and all $p\ge 2$, $\norm{\inr{X,t}}_{L_p}\leq L \sqrt{p}\norm{\inr{X, t}}_{L_2}$ 
	\item[iv)]  there exist $u_0>0$ and $\beta_0>0$ such that for all $t\in\R^d$, 
	\[\bP\left[|\inr{X, t}|\geq u_0\norm{\inr{X, t}}_{L_2}\right]\geq \beta_0\enspace.\]  
\end{enumerate}
The LASSO estimator, defined by
\begin{equation*}
\tilde t \in\argmin_{t\in\R^d}\left(\frac{1}{N}\sum_{i=1}^N \left(Y_i-\inr{X_i, t}\right)^2 +c_1\norm{\zeta}_{L_{q_0}}\sqrt{\frac{\log(ed)}{N}}\norm{t}_1\right)
\end{equation*}satisfies for every $1\leq p\leq 2$, 
\begin{equation*}
\norm{\tilde t - t^*}_p\leq c_4(L, u_0, \kappa_0)\norm{\zeta}_{L_{q_0}}s^{1/p}\sqrt{\frac{\log(ed)}{N}}\enspace,
\end{equation*}
 with probability at least
\begin{equation}\label{eq:proba_lasso_classic}
 1-\frac{c_2 \log^{q_0}N}{N^{q_0/2-1}} - 2 \exp\left(-c_3s \log(ed/s)\right)\enspace.
 \end{equation}\end{Theorem}
This paper shows that Theorem~\ref{theo:lasso_classi} holds for a MOM version of  the LASSO estimator under much weaker assumptions, with a better probability estimate than \eqref{eq:proba_lasso_classic}. More precisely, the following theorem is proved. 
\begin{Theorem}\label{theo:mom_lasso}
Assume that  $t^*$ is $s$-sparse,  $N\geq c_0 s \log(ed/s)$, $X$ is isotropic and  
\begin{enumerate}
	\item[i')] $|\cI|\geq N/2$ and $|\cO|\leq c_1 s \log(ed/s)$ (the number of outliers may be proportional to the sparsity times $\log(ed/s)$),
	\item[ii)] $ \zeta=Y-\inr{X, t^*} \in L_{q_0}$ for some $q_0>2$
	\item[iii')] for every $1\leq p\leq C_0 \log(ed)$, $\norm{\inr{X,e_j}}_{L_p}\leq L \sqrt{p}\norm{\inr{X,e_j}}_{L_2}$ where $(e_j)_{j\in[d]}$ is the canonical basis of $\R^d$ and $C_0$ is some absolute constant,
	\item [iv')] there exists $\theta_0$ such that  $\norm{\inr{X, t}}_{L^1}\leq \theta_0\norm{\inr{X, t}}_{L^2}$, for all $t\in \R^d$,
	\item [v)] there exists $\theta_m$ such that ${\rm var}(\zeta\inr{X, t})\leq \param{m}^2 \norm{t}_2^2$, for all $t\in\R^d$.
\end{enumerate}
There exists an estimator $\hat t$, called MOM-LASSO, satisfying  for every $1\leq p\leq 2$, 
\begin{equation*}
\norm{\hat t - t^*}_p\leq c_4(L, \theta_m)\norm{\zeta}_{L_{q_0}}s^{1/p}\sqrt{\frac{1}{N}\log\left(\frac{ed}{s}\right)}\enspace,
\end{equation*}
with probability at least 
\begin{equation}\label{eq:proba_mom_lasso}
1-c_2 \exp(-c_3 s \log(ed/s))\enspace.
\end{equation}
\end{Theorem}
Theoretical properties of MOM LASSO outperform those of LASSO in several ways.
%Compare Theorem~\ref{theo:lasso_classi} and Theorem~\ref{theo:mom_lasso}. 
\begin{itemize}
	\item Estimation rates achieved by MOM-LASSO are the actual minimax rates $s \log(ed/s)/N$,  see \cite{BLT16}, while classical LASSO estimators achieve the rate $s \log(ed)/N$. This improvement is possible thanks to the adaptation step in MOM-LASSO. 
	\item  the probability deviation in  \eqref{eq:proba_lasso_classic} is polynomial -- $1/N^{(q_0/2-1)}$ in \eqref{eq:proba_lasso_classic} -- it is exponentially small for MOM LASSO.  Exponential rates for LASSO hold only if $\zeta$ is subgaussian ($\norm{\zeta}_{L_p}\leq C \sqrt{p}\norm{\zeta}_{L_2}$ for all $p\geq2$).
%	, the probability deviation for LASSO is the one of MOM LASSO $1-c_2 \exp(-c_3 s \log(ed/s))$. But this requires the noise to be sub-gaussian in Theorem~\ref{theo:lasso_classi} whereas such a result is achieved under the only $L_{q_0}$ for $q_0>2$ assumption on $\zeta$ in Theorem~\ref{theo:mom_lasso},
	\item MOM LASSO is insensitive to data corruption by up to $s$ times $\log(ed/s)$ outliers while only one outlier can be responsible of a dramatic breakdown of the performances of LASSO.
	\item All assumptions on $X$ are weaker for MOM LASSO than for LASSO. In particular, condition \textit{v)} holds with $\param{m} = \norm{\zeta}_{L_4}$ if for all $t\in\R^d$, $\norm{\inr{X, t}}_{L_4}\leq \param{0} \norm{\inr{X, t}}_{L_2}$ -- which is a much weaker requirement than condition iii) for LASSO. 
\end{itemize}
%:  
%  
%
%

From a mathematical point of view, our results are based on a slight extension of the Small Ball Method (SBM) of \cite{MR3431642,Shahar-COLT} to handle non-i.d. data. SBM is also extended to bound both quadratic and multiplier parts of the quadratic loss. Otherwise, all arguments are standard, which makes the approach very attractive and easily reproducible in other frameworks of statistical learning.

The paper is organized as follows. Section~\ref{sec:setting} briefly presents the general setting and our main illustrative example. Section~\ref{Sec:LFT} presents Le Cam's construction of estimators based on tests. We also show why many learning procedures may be obtained by this approach. The construction of estimators and the main assumptions are gathered in Section~\ref{sec:EstimatorsAssumptions}. Our main theorems are stated in Section~\ref{sub:main_result} and proved in Section~\ref{sec:Proofs}.

\paragraph{Notation} For any real number $x$, let $\lfloor x\rfloor$ denote the largest integer smaller than $x$ and let $[x]=\{1,\ldots,\lfloor x\rfloor\}$ if $x\ge 1$. For any finite set $A$, let $|A|$ denote its cardinality. All along the paper, $(\cabs{i})_{i\in \N}$ denote absolute constants which may vary from line to line and $\param{\cdot}$, with various subscripts, denote real valued parameters introduced in the assumptions. Finally, for any set $\mathcal{G}$ for which it makes sense, for any $g\in\mathcal{G}$, $c\ge 0$ and $\mathcal{C}\subset\mathcal{G}$, 
\[
g+c\mathcal{C}=c\cC+g=\{h : \exists g'\in \mathcal{C} \text{ such that }h=g+cg'\}\enspace.
\]
Let also $g+\cG=g+1\cG$. We also denote by $I(g\in\cC)$ the indicator function of the set $\cC$ which equals to $1$ when $g\in\cC$ and $0$ otherwise.

\section{Setting}\label{sec:setting}
Let $\mathcal X$ denote a measurable space and let $(X,Y),(X_i,Y_i)_{i\in[N]}$ denote random variables taking values in $\cX\times \R$, with respective distributions $P, (P_i)_{i\in[N]}$. Given a probability distribution $Q$, let $L^2_Q$ denote the space of all functions $f$ from $\cX$ to $\R$ such that $\norm{f}_{L^2_Q}<\infty$ where $\norm{f}_{L^2_Q} = \big(Q f^2\big)^{1/2}$. Let $F\subset L^2_{P}$ denote a convex class of functions $f:\cX\to\R$. Assume that $PY^2<\infty$ and let, for all $f\in F$,
\[
R(f)=P\big[(Y-f(X))^2\big],\quad \bayes\in\argmin_{f\in F}R(f) \mbox{ and }  \zeta=Y-\bayes(X)\enspace.
\]
%We want to estimate $\bayes$ based on $(X_i,Y_i)_{i\in[N]}$. 
Let $\norm{\cdot}$ denote a norm defined onto a linear subspace $E$ of $L^2_P$ containing $F$.
%, that will be used as a regularization function.

\paragraph*{Example : $\ell_1$-regularization of linear functionals}
%As a proof of concept, we apply our results to $\ell_1^d$ regularization of linear functionals, which has been one of the most studied regularization method  in high-dimensional statistics during the last decade (cf. \cite{MR3307991,MR2807761}). For this problem, 
For every $t=(t_j)_1^d\in\R^d$ and $1\leq p\leq +\infty$, let 
\begin{equation*}
F = \{\inr{\cdot, t}: t\in\R^d\}\;  \an \; \norm{\inr{\cdot, t}} = \norm{t}_1,\; \text{where}\; \norm{t}_p=\left(\sum_{j=1}^d |t_j|^p\right)^{1/p}\enspace.
\end{equation*}
Let $\bayes=\inr{\cdot, t^*}\in F$, where
\begin{equation*}
t^*\in\argmin_{t\in\R^d}\left\{P\big(Y-\inr{X, t}\big)^2\right\}\enspace.
\end{equation*}
Whenever it's necessary, $(e_1, \ldots, e_d)$ will denote the canonical basis of $\R^d$ and $B_p^d$ (resp. $S_p^{d-1}$) will denote the unit ball (resp. sphere) associated to $\norm{\cdot}_p$. 
%The aim of this example is to compute various parameters introduced in the general framework and to show typical results that follow from our analysis in a well studied case. 
To ease readability in this example, we focus on rates of convergence, we do not consider the ``full'' non-i.i.d. setup and assume that $P=P_i$ for all $i\in\cI$. We write $L^q$ for $L^q_P$ to shorten notations. 

\section{Learning from tests}\label{Sec:LFT}
\subsection{General Principle}
This section details the ideas underlying the construction of a MOM estimator using an extension of Le Cam's approach.
\paragraph{Basic idea}
By definition of the oracle $f^*$, one has
\[
\bayes=\argmin_{f\in F}R(f)=\argmin_{f\in F}\sup_{g\in F}\{R(f)-R(g)\},\; \text{where} \; R(f)=P[(Y-f(X)^2]\enspace.
\]
As $T_{\text{id}}(g,f)=R(f)-R(g)$ depends on $P$, we estimate it by test statistics $T(g,f,(X_i,Y_i)_{i\in[N]})\equiv T_N(g,f)$ that is, real random variables such that 
\begin{equation}\label{eq:DefTest}
T_N(f,g)+T_N(g,f)=0\enspace. 
\end{equation}
These statistics are used to \emph{compare} $f$ to $g$, simply by saying that $g$ $T_N$-\emph{beats} $f$ iff $T_N(g,f)\ge 0$. In this paper, the statistics $T_N(g,f)$ are median-of-means estimators of $R(f)-R(g)$ (cf. \eqref{eq:beat_on_Bk} in Section~\ref{Sec:QOMProcesses}).

\paragraph{Le Cam's construction}
%Given a collection of test statistics, we shall then \emph{aggregate} these tests using Birgé's $T$-aggregation principle (see for example \cite{MR2219712} and the references therein), which that derives from Le Cam \cite{MR0334381}. 
Let $(T_N(g,f))_{f,g\in F}$ denote a collection of test statistics and let $d(\cdot, \cdot)$ denote a pseudo-distance on $F$ measuring (or related to) the risk we want to control. Let for all $f\in F$,
%as a criterion the "radius" $C_{T_N}(f)=\max_{g\in\mathcal B_{T_N}(f)}d(f,g)$ of $\mathcal{B}_{T_N}(f)$ for some distance $d$ and as an estimator any minimizer $\hat f_{T_N}$ of this criterion. Given such pairwise comparisons, one can define
\[
\mathcal{B}_{T_N}(f)=\{g\in F : T_N(g,f)\ge 0\}
\]be the set of all functions $g\in F$ that beat $f$. If $f$ is far from $f^*$, then $\mathcal{B}_{T_N}(f)$ is expected to have a large radius w.r.t. $d(\cdot, \cdot)$. We therefore introduce this radius as a criteria to minimize : for all $f\in F$, let $C_{T_N}(f) = \sup_{g\in \cB_{T_N}(f)}d(f,g)$. 

By \eqref{eq:DefTest}, $f\in \mathcal{B}_{T_N}(g)$ or $g\in \mathcal{B}_{T_N}(f)$ (both happen if $T_N(f,g)=0$), hence $d(f,g)\le C_{T_N}(f)\vee C_{T_N}(g)$. In particular, for all $f\in F$, 
\begin{equation}\label{GenRiskBound2}
d(f,\bayes)\le C_{T_N}(f)\vee C_{T_N}(\bayes)\enspace. 
\end{equation}
Eq~\eqref{GenRiskBound2} suggests to define the estimator
\begin{equation}\label{Def:Testimators}
 \hat f_{T_N}\in \argmin_{f\in F}C_{T_N}(f)=\argmin_{f\in F}\sup_{g\in \cB_{T_N}(f)}d(f,g)\enspace.
\end{equation}
This estimator satisfies, from Eq~\eqref{GenRiskBound2},
\begin{equation}\label{GenRiskBound}
d(\hat f_{T_N},\bayes)\le C_{T_N}(\bayes)\enspace.
\end{equation}
Risk bounds for $\hat f_{T_N}$ follow from \eqref{GenRiskBound} and upper bounds on the radii of $\mathcal{B}_{T_N}(\bayes)$. 

\begin{Remark}
More generally, one can compare only the elements of a subset $\mathcal F\subset F$, typically a maximal $\epsilon$-net by introducing for all $f\in \mathcal F $, the set
\begin{equation}\label{def:BGen}
\mathcal{B}_{T_N}(f,\mathcal F)=\{g\in \mathcal F : T_N(g,f)\ge 0\} 
\end{equation} and then by minimizing the diameter of $\mathcal{B}_{T_N}(f,\mathcal F)$ over $\cF$. This usually improves the rates of convergence for constant deviation results when there is a gap in Sudakov's inequality of the localized sets of $F$ (cf. Section~5 in \cite{LM13} for more details). These results are not presented because we are interested in exponentially large deviation results for which our results are optimal. 
\end{Remark}

\paragraph{Dealing with regularization : the link function}
Statistical performances of estimators and the radius of $\mathcal{B}_{T_N}(\bayes)$ can be measured by two norms: the regularization norm $\|\cdot\|$ and $\|.\|_{L^2_P}$. As \eqref{Def:Testimators} allows only for one distance $d$, we propose the following extension of Le Cam approach to handle two metrics. 

To introduce this extension, assume first that $d(f,g)=\|f-g\|_{L^2_{P}}$ can be computed for all $f, g\in F$ (this is the case if the distribution of the design is known). The next paragraph explains how to deal with the more common framework where this distance is unknown. Remark that
\[
C_{T_N}(f)=\sup_{g\in\mathcal B_{T_N}(f)}\|f-g\|=\min\left\{\rho\ge 0 : \sup_{g\in \mathcal B_{T_N}(f)}\|g-f\|\le \rho\right\}\enspace.
\]
%Theorem \ref{theo:basic-combining-loss-and-reg} shows that a minimizer $\hat f^{(1)}$ of $C_{T_N}$ satisfies : $\norm{\hat f^{(1)}-\bayes}$ is controlled but $d(\hat f^{(1)},\bayes)$ can be large.
%To control also $d(\hat f^{(1)},\bayes)$, we use a function 
The main point to extend Le Cam's approach to simultaneously control two norms is to design a link function $r(\cdot)$. In a nutshell, the values $r(\rho)$ is the $L^2_P$-minimax rate of convergence in a ball of radius $\rho$ for the regularization norm
%Now let $r:\R^+\to\R^+$ 
%%($r$ will be defined in ) that basically maps 
%denote a function mapping the optimal rate of convergence in the reference norm $\|\cdot\|$ to the optimal rate for $d$ 
(cf. \eqref{eq:defr} in Section~\ref{sec:Complexity} for a formal definition). Then one can define
\begin{equation*}
C_{T_N}^{(2)}(f)=\min\left\{\rho\ge 0 : \sup_{g\in \mathcal B_{T_N}(f)}\|g-f\|\le \rho\text{ and } \sup_{g\in \mathcal B_{T_N}(f)}d(f,g)\le r(\rho)\right\}\enspace. 
\end{equation*}
Theorem \ref{theo:basic-combining-loss-and-reg} shows that while a minimizer $\hat f^{(1)}$ of $C_{T_N}$ has only a nice risk for $\norm{\cdot}$, a minimizer $\hat f^{(2)}$ of $C_{T_N}^{(2)}$ has both $\norm{\hat f^{(2)}-\bayes}$ and $d(\hat f^{(2)},\bayes)$ properly controlled.

\paragraph{Dealing with unknown norms : the isometry property}In general, $L_P^2$-distances cannot be directly computed and have to be estimated. To deal with this issue, one considers usually the empirical $L^2_{P_N}$ distance and prove that empirical and actual distances are equivalent outside a $L^2_P$-ball centered in $f^*$ (cf. for instance, remark after Lemma~2.6 in \cite{LM13}). Unfortunately this approach only works under strong concentration property that we want to relax in this paper. 

The unknown $L_P^2$-metric is instead estimated by a median-of-means approach, that is, we use MOM estimators $d_N(f,g)$ of all $d(f,g)$ (cf. Section~\ref{sec:estimators}). The final estimator is therefore defined as a minimizer of 
\begin{equation*}
C_{T_N}^{\prime\prime}(f)=\min\left\{\rho\ge 0 : \sup_{g\in \mathcal B_{T_N}(f)}\|g-f\|\le \rho\text{ and } \sup_{g\in \mathcal B_{T_N}(f)}d_N(f,g)\le r(\rho)\right\}\enspace. 
\end{equation*}
% It appears that the argument leading to a control of the s the  \eqref{GenRiskBound} also imply a control of 
% $d_N(\hat f,\bayes)$. Upper bounds on $d(\hat f,\bayes)$ then rely on an isometry property between $d_N$ and $d$
% %, implying $d_N(\hat f,\bayes)\gtrsim d(\hat f,\bayes)$ 
% holding with large probability (cf. Lemma~\ref{lem:Isometry}).
%meaning with a constant $1$ in front of the bias term instead of the constants $1+\epsilon$ that could be derived from the inequality $(a+b)^2\le (1+\epsilon)a^2+(1+\epsilon^{-1})b^2$ and the bound on the excess risk norm. 

\subsection{Examples}\label{sec:Example}
Le Cam's approach has been used by Birgé to define $T$-estimators (cf. \cite{MR2449129,MR2219712,MR3186748}) and by Baraud, Birgé and Sart to define $\rho$-estimators (cf. \cite{MR3565484,BaraudBirgeSart}).  \cite{MR2834722,MR3224300} also built efficient estimator selection procedures with this approach. It also extends many common procedures in statistical learning theory, as shown by the following examples.

\paragraph{Example 1 : Empirical minimizers} 
Assume $T_N(g,f)=\ell_N(f)-\ell_N(g)$ for some random function $\ell_N:F\to\R$ and denote by $\hat f=\arg\min_{f\in F}\ell_N(f)$ a minimizer of the corresponding criterion (provided that it exists and is unique). 
Then it is easy to check that $\mathcal B_{T_N}(\hat f)=\{\hat f\}$, so its radius is null, while the radius of any other point $f$ is larger than $d(f,\hat f)>0$ (whatever the non-degenerate notion of pseudo-distance used for $d$). 
It follows that $\hat f$ is the estimator \eqref{Def:Testimators}. 
In particular, any possibly penalized empirical risk minimizer 
\[
\hat f=\arg\min_{f\in F}\{P_N\ell_f+{\rm reg}(f)\}
\] 
is obtained by Le Cam's construction with the tests 
\[T_N(g,f)=P_N(\ell_f-\ell_g)+{\rm reg}(f)-{\rm reg}(g)\enspace.\]
These examples encompass classical empirical risk minimizers of \cite{MR1641250} but also their robust versions from \cite{MR0161415,MR2906886}.
%, where $\ell_f(x,y)=(y-f(x))^2$ in this paper. 

\paragraph{Example 2 : median-of-means estimators}
Another, perhaps less obvious example is the median-of-means estimator \cite{MR1688610, MR855970, MR702836} of the expectation $PZ$ of a real valued random variable $Z$. Let $Z_1,\ldots,Z_N$ denote a sample and let $B_1,\ldots,B_K$ denote a partition of $[N]$ into bins of equal size $N/K$. The estimator $\text{MOM}_K(Z)$ is the (empirical) median of the vector of empirical means $\left( P_{B_k}Z=|B_k|^{-1}\sum_{i\in B_k}Z_i\right)_{k\in[K]}$. Recall that 
\[
PZ=\argmin_{m\in\R}P(Z-m)^2=\argmin_{m\in\R}\max_{m'\in\R}P[(Z-m)^2-(Z-m')^2]\enspace.
\]
Define the MOM test statistic to compare any $m,m'\in \R$ by
\begin{align*}
T_N(m,m')&=\text{MOM}_K[(Z-m')^2-(Z-m)^2]\enspace.
\end{align*}
Basic properties of the median (recalled in Eq~\eqref{prop:Cone} and \eqref{prop:Opposes} of Section~\ref{Sec:QOMProcesses}) yield
\begin{align*}
 T_N(m,m')&=(m')^2-m^2+\text{MOM}_K[-2Z(m'-m)]\\
&=(m')^2-2m'\text{MOM}_K(Z)-[m^2-2m\text{MOM}_K(Z)]\\
&=(m'-\text{MOM}_K(Z))^2-(m-\text{MOM}_K(Z))^2\enspace. 
\end{align*}
Defining $\ell_N(m)=(m-\text{MOM}_K(Z))^2$, one has
\[
T_N(m,m')=\ell_N(m')-\ell_N(m)\enspace.
\]
As in the previous example, Le Cam's estimator based on $T_N$ is therefore the unique minimizer of $\ell_N$, that is $\text{MOM}_K(Z)$.

\paragraph{Example 3 : ``Champions" of a Tournament}\label{Ex:LugMen} \cite{LugosiMendelson2016} introduced median-of-means tournaments. More precisely, they used median-of-means tests to compare elements in $F$. These tests cannot be separated $T_N(f,g)\neq\ell_N(g)-\ell_N(f)$ in general. \cite{LugosiMendelson2016} assume that an upper bound $r^*$ on the radius $C_{T_N}(\bayes)$ of $\mathcal B_{T_N}(\bayes)$  (that holds with exponentially large probability) is known from the statistician and call ``champion" any element $\hat f$ of $F$ such that $C_{T_N}(\hat f)\le r^*$. 
%They prove that any champion has radius smaller than $r^*$. 
It is clear that, by definition the radius $C_{T_N}(\hat f_{T_N})$ of $\hat f_{T_N}$ is smaller than $C_{T_N}(\bayes)$ and therefore smaller than $r^*$. This means that $\hat f_{T_N}$ is a ``champion" for this terminology. The main advantage of Le Cam's approach is that $r^*$ (which usually depends on some attribute of the oracle like the sparsity) is not required to \emph{build} the estimator $\hat f_{T_N}$.

\section{Construction of the regularized MOM estimators}\label{sec:EstimatorsAssumptions}
\subsection{Quantile of means processes and median-of-means tests}\label{Sec:QOMProcesses}
This section presents median-of-means (MOM) tests used in this work. Designing a family of tests $(T_N(g, f): f, g\in F)$ is one of the most important building blocks in Le Cam's approach together with the right choice of the metric measuring the diameters $\cB_{T_N}(f)$ for $f\in F$.

Start with a few notations. For all $\alpha\in[0,1]$, $\ell\ge 1$ and $z\in\R^\ell$, the set of $\alpha$-quantiles of $z$ is denoted by
\[
\mathcal Q_{\alpha}(z)=\left\{x\in \R : \frac1\ell\sum_{k=1}^\ell I(z_i\le x)\ge \alpha\quad\text{and}\quad\frac1\ell\sum_{k=1}^\ell I(z_i\ge x)\ge 1-\alpha\right\}\enspace.
\]
For a non-empty subset $B\subset [N]$ and a function $f:\cX\times\R\to\R$, let 
\begin{equation*}
P_Bf=\frac{1}{|B|}\sum_{i\in B}f(X_i,Y_i) \mbox{ and } \overline{P}_Bf=\frac{1}{|B|}\sum_{i\in B} P_i f\enspace.
\end{equation*} Let $K\in[N]$ and let $(B_1,\ldots,B_K)$ denote an equipartition of $[N]$ into bins of size $|B_k|=N/K$. When $K$ does not divide $N$, at most $K-1$ data can be removed from the dataset. For any real number $\alpha\in [0,1]$ and any function $f:\cX\times\R\to\R$, the set of $\alpha$-quantiles of empirical means is denoted by
\[
\mathcal Q_{\alpha,K}(f)=\mathcal Q_{\alpha}\left((P_{B_k}f)_{k\in[K]}\right)\enspace.
\]
With a slight abuse of notations, we shall repeatedly denote by $Q_{\alpha,K}(f)$ any element in $\mathcal Q_{\alpha,K}(f)$ and write $Q_{\alpha,K}(f)=u$ if $u\in \mathcal Q_{\alpha,K}(f)$, $Q_{\alpha,K}(f)\ge u$ if $\sup\mathcal Q_{\alpha,K}(f)\ge u$, $Q_{\alpha,K}(f)\le u$ if $\inf\mathcal Q_{\alpha,K}(f)\le u$, and $Q_{\alpha,K}(f)+Q_{\alpha',K}(f')$ any element in the Minkowski sum $\mathcal Q_{\alpha,K}(f)+\mathcal Q_{\alpha',K}(f')$. Let also $\text{MOM}_K(f)=Q_{1/2,K}(f)$ denote an empirical median of the empirical means on the blocks $B_k$. Empirical quantiles satisfy for any $ c\ge 0$, $f,f':\cX\times\R\to\R$ and $\alpha\in[0,1]$,
\begin{gather}
\mathcal Q_{\alpha,K}(c f)=c \mathcal Q_{\alpha,K}(f)\enspace,\label{prop:Cone}\\
\mathcal Q_{\alpha,K}(-f)=-\mathcal Q_{1-\alpha,K}(f)\enspace,\label{prop:Opposes}\\
\sup \cro{{\mathcal Q}_{1/4,K}(f)+{\mathcal Q}_{1/4,K}(f')}\le \inf \mathcal Q_{1/2,K}(f+f')\enspace,\\
 \sup \mathcal Q_{1/2,K}(f+f')\le \inf \cro{{\mathcal Q}_{3/4,K}(f)+{\mathcal Q}_{3/4,K}(f')}\enspace.
\label{prop:Sum}
 \end{gather}
With some abuse of notations, we shall write these properties respectively
\begin{align*}
& Q_{\alpha,K}(c f)=c Q_{\alpha,K}(f),\qquad Q_{\alpha,K}(-f)=-Q_{1-\alpha,K}(f)\enspace,\\
 Q_{1/4,K}(f)&+Q_{1/4,K}(f')\le \MOM{K}{f+f'}\le Q_{3/4,K}(f)+Q_{3/4,K}(f')\enspace.
\end{align*}
A \emph{regularization} parameter $\lambda>0$ is introduced to balance between data adequacy and regularization. 
The (quadratic) loss and regularized (quadratic) loss are respectively defined on $F\times\cX\times\R$ as the real valued functions such that  
 \begin{equation*}
 \ell_f(x,y) = (y-f(x))^2, \quad \ell^\lambda_f = \ell_f+\lambda \norm{f} ,\qquad \forall (f,x,y)\in F\times\cX\times\R\enspace.
 \end{equation*}
To compare/test functions $f$ and $g$ in $F$, median-of-means tests between $f$ and $g$ are now defined by
\begin{equation}\label{eq:beat_on_Bk}
T_{K,\lambda}(g,f)=\MOM{K}{\ell^\lambda_f - \ell^\lambda_g}=\MOM{K}{\ell_f-\ell_g}+\lambda(\|f\|-\|g\|)\enspace.
\end{equation}
From \eqref{prop:Opposes}, $T_{K,\lambda}$ satisfies \eqref{eq:DefTest} and is a tests statistic in the sense of Section \ref{Sec:LFT}. 
   
\subsection{Main assumptions}
Recall that $[N]=\cO\cup\cI$ and that $(X_i,Y_i)_{i\in \cO}$ is a set of outliers on which we make no assumption so these may be aggressive in any sense one can imagine. The remaining informative data $(X_i,Y_i)_{i\in \cI}$ need to bring enough information onto $f^*$. We therefore need some assumption on the sub-dataset $(X_i,Y_i)_{i\in \cI}$ and, in particular, some connexion between the distributions $P_i$ for $i\in\cI$ and $P$. These assumptions are pretty weaksince we only assume essentially that the $L^2_P, L^2_{P_i}$ and $L^1_{P_i}$ geometries are comparable in the following sense. 
%
%The first assumption bounds the moments of order $2$ of all functions in $F$ for all distributions $(P_{i})_{i\in I}$ by those of $P$. 

\begin{Assumption}\label{ass:Mom2F}
There exists $\param{r}\ge 1$ such that, for all $i\in\cI$ and $f\in F$, 
\[
\norm{f-\bayes}_{L^2_{P_i}}\le \param{r}\norm{f-\bayes}_{L^2_P}\enspace.
\] 
\end{Assumption}
Of course, Assumption~\ref{ass:Mom2F} holds in the i.i.d. framework, with $\param{r}=1$ and $\cI=[N]$.
The second assumption bounds the correlation between the noise function $\zeta:(y, x)\in\R\times \cX\to y-\bayes(x)$ and the design on the shifted class $F-\bayes$ in $L^2_Q$ for all $Q\in \{P,(P_i)_{i\in \cI}\}$.

\begin{Assumption}
%[Margin condition]
\label{ass:margin}
There exists $\param{m}>0$ such that, for all $Q\in \{P,(P_i)_{i\in \cI}\}$ and $f\in F$,
\[
 \text{var}_{Q}(\zeta(f-\bayes))=Q\left[\zeta^2(f-\bayes)^2-[Q(\zeta(f-\bayes))]^2\right]\le \param{m}^2\norm{f-\bayes}^2_{L^2_P}\enspace.
\]
\end{Assumption}
Let us give some examples where Assumption~\ref{ass:margin} holds. If the noise random variable $\zeta(Y, X)$ (resp. $\zeta(Y_i, X_i)$ for $i\in\cI$) has a variance conditionally to $X$ (resp. $X_i$ for $i\in\cI$) that is uniformly bounded then Assumption~\ref{ass:margin} holds. This is, for example, the case, when $\zeta(Y,X)$ (resp. $\zeta(Y_i, X_i)$ for $i\in\cI$) is independent of $X$ (resp. $X_i$ for $i\in\cI$) and has finite $L^2$-moment  with $\param{m}=\max_{Q\in P,\{P_i\}_{i\in \cI}}\norm{\zeta}_{L^2_Q}$. It also holds without independence under higher moment conditions.
For example, assume $\sigma=\max_{Q\in P,\{P_i\}_{i\in \cI}}\norm{\zeta}_{L^4_Q}<\infty$ and, for every $f\in F$, $\norm{f-\bayes}_{L^4_Q}\leq \theta_1 \norm{f-\bayes}_{L^2_P}$ then by Cauchy-Schwarz inequality,
$\sqrt{{\rm var}_Q(\zeta(f-\bayes))} \leq \norm{\zeta(f-\bayes)}_{L^2_Q}\leq \norm{\zeta}_{L^4_Q}\norm{f-\bayes}_{L^4_Q}\leq \theta_1 \sigma\norm{f-\bayes}_{L^2_P}$ and so Assumption~\ref{ass:margin} holds for $\param{m} = \theta_1 \sigma$. 

\begin{Assumption}\label{ass:small-ball}
There exists $\theta_0\ge 1$ such that for all $f\in F$ and all $i\in \cI$
\begin{equation*}
\norm{f-f^*}_{L^2_P}\le \theta_0\norm{f-f^*}_{L^1_{P_i}}\enspace.
\end{equation*}
\end{Assumption}
By Cauchy-Schwarz inequality, $\norm{f-f^*}_{L^1_{P_i}}\leq \norm{f-f^*}_{L^2_{P_i}}$ for all $f\in F$ and $i\in\cI$. Therefore, Assumptions~\ref{ass:Mom2F} and~\ref{ass:small-ball} together imply that all norms $L^2_P, L_{P_i}^2, L_{P_i}^1, i\in\cI$ are equivalent over $F-f^*$. Note also that Assumption~\ref{ass:small-ball} is related to the small ball property (cf. \cite{MR3431642,Shahar-COLT}) as shown by Proposition~\ref{prop:SBP=L2/L1}  bellow. The small ball property has been recently used in Learning theory and signal processing. We refer to \cite{MR3431642,LM_compressed,shahar_general_loss,MR3364699,Shahar-ACM,RV_small_ball} for examples of distributions satisfying this assumption.
% (and therefore Assumption~\ref{ass:small-ball} when $L^2_P$ and $L^2_{P_i}$ are equivalent on $F-f^*$).  

\begin{Proposition}\label{prop:SBP=L2/L1} Let $Z$ be a real-valued random variable. 
\begin{enumerate}
	\item If there exist $\kappa_0$ and $u_0$ such that $\bP(|Z|\geq \kappa_0\norm{Z}_2)\ge u_0$ then $\norm{Z}_2\le (u_0\kappa_0)^{-1}\norm{Z}_1$.
	\item If there exists $\param{0}$ such that $\norm{Z}_2\le \param{0}\norm{Z}_1$, then for any $\kappa_0<\param{0}^{-1}$, $\bP(|Z|\geq \kappa_0\norm{Z}_2)\ge u_0$ where $u_0=(\param{0}^{-1}-\kappa_0)^2$.
\end{enumerate} 
\end{Proposition}
\begin{proof} If $\bP(|Z|\geq \kappa_0\norm{Z}_2)\ge u_0$ then 
 \begin{equation*}
  \norm{Z}_{1}\ge \int_{|z|\ge \kappa_0\norm{Z}_{2}}|z|P_Z(dz)\ge u_0\kappa_0\norm{Z}_{2}\enspace,
 \end{equation*}where $P_Z$ denotes the distribution of $Z$. Conversely, if $\norm{Z}_2\leq \param{0} \norm{Z}_1$, the Paley-Zigmund's argument \cite[Proposition~3.3.1]{MR1666908} shows that, if $p=\bP\left(|Z| \geq \kappa_0 \norm{Z}_{2}\right)$,
\begin{align*}
\norm{Z}_{2}&\le \param{0}\norm{Z}_{1} =  \param{0} \pa{\E[|Z|I(|Z| \leq \kappa_0 \norm{Z}_{2})] + \E[|Z|I(|Z| \geq \kappa_0 \norm{Z}_{2})]}\\
&\le \param{0}\norm{Z}_{2}\pa{\kappa_0 + \sqrt{p}}\enspace.
\end{align*}
As one can assume that $\norm{Z}_2\ne 0$, $p\geq (\param{0}^{-1}-\kappa_0)^2$.
 \end{proof}

\subsection{Complexity parameters and the link function}\label{sec:Complexity}
This section defines the link function $r(\cdot)$ making the connections between norms that will be required in the extension of Le Cam's approach to a simultaneous control of two norms (one of the two being unknown). For any $\rho\ge 0$ and any $f\in E$, let
\begin{gather*}
% B_2(f,r)=\{g\in L^2_P : \|f-g\|_{L^2_P}\le r\}, \quad S_2(f,r)=\{g\in L^2_P : \|f-g\|_{L^2_P}= r\}\enspace,\\
 B(f,\rho) = \{g\in E : \norm{f-g}\leq \rho\},\qquad S(f,\rho) = \{g\in E : \norm{g-f} = \rho\}
%,\qquad \rho B=B(0,\rho),\qquad \rho S=S(0,\rho)
\enspace.
\end{gather*}

%
%To state our main results, let us introduce the following complexity parameters.

%\textcolor{red}{Matthieu : I think it would be better to take the sum over most $i\in [N]$ rather than all $i\in [N]$, at least for the first fixed point where only the geometry of $L_2(P)$ on $F$ is involved. Of course, this implies some tedious computations of constants, but I think that it is in the spirit of the robustness results we have...}
\begin{Definition}\label{def:the-three-paraemters}
Let $(\eps_i)_{i\in \cI}$ be independent Rademacher random variables, independent from $(X_i,Y_i)_{i\in\cI}$ and let $\cJ=\{J\subset \cI, |J|\ge |\cI|/2\}$. For any $\gamma_Q,\gamma_M>0$ and $\rho>0$ let $F_{f^\star,\rho,r}=\{f \in F \cap B(f^\star,\rho) : \norm{f-f^\star}_{L^2_P} \leq r\}$,
\begin{align*}
 \fQ^{\gamma_Q}_{f^\star,\rho}&=\left\{r>0: \forall J\in \cJ,\;\E \sup_{f \in F_{f^\star,\rho,r}} \left|\sum_{i\in J} \eps_i (f-f^\star)(X_i)\right| \leq \gamma_Q |J| r \right\}\enspace,\\
 \fM^{\gamma_M}_{f^\star,\rho}&=\left\{r>0: \forall J\in \cJ,\;\E \sup_{f \in F_{f^\star,\rho,r}} \left|\sum_{i\in J} \eps_{i} (Y_i-f^\star(X_i)) (f-f^\star)(X_i)\right| \leq \gamma_M |J| r^2 \right\}
\end{align*}and the two fixed point functions
\begin{gather*}
r_Q(\rho,\gamma_Q)=
\sup_{f^\star\in F}\{\inf \fQ^{\gamma_Q}_{f^\star,\rho}\},\qquad r_M(\rho,\gamma_M) =
\sup_{f^\star \in F}\{\inf \fM^{\gamma_M}_{f^\star,\rho}\}\enspace.
\end{gather*}The \textbf{link function} is any continuous and non-decreasing function $r:\R_+\to \R_+$ such that for all $\rho>0$
\begin{equation}\label{eq:defr}
r(\rho) = r(\rho, \gamma_Q, \gamma_M)\geq \max(r_Q(\rho, \gamma_Q), r_M(\rho, \gamma_M)).
\end{equation}
\end{Definition}Note that if the function $\rho\to \max(r_Q(\rho, \gamma_Q), r_M(\rho, \gamma_M))$ is itself continuous and non-decreasing then it can be taken equal to $r(\cdot)$. In the next paragraph, we provide an explicit computation of the functions $r_Q(\cdot)$, $r_M(\cdot)$ and $r(\cdot)$ in the ``LASSO case''.

\paragraph*{Complexity parameters for the $\ell_1$-regularization}
%Explicit computations of  follow from classical results on the expectation of symmetrized empirical processes. For instance,
%, a useful result is Theorem~1.6 in \cite{shahar_gafa_ln} which 
One can derive $r_Q(\cdot)$ and $r_M(\cdot)$ from Gaussian mean widths defined
%
%We introduce the Gaussian mean width of a subset in $\R^d$ which is appropriate for the particular case of $\ell_1^d$-regularization that we have been using as a proof of concept so far: 
for any $V\subset\R^d$, by
\begin{equation}\label{eq:Gauss_mean_width}
\ell^*(V) = \E\cro{ \sup_{(v_j)\in V} \sum_{j=1}^d g_j v_j},\quad \text{where }\quad (g_1,\ldots,g_d)\sim\mathcal N_d(0,I_d)\enspace.
\end{equation}
%
%Now, the functions $r_Q(\cdot)$ and $r_M(\cdot)$ in the $\ell_1^d$-regularization example are computed in \cite[Theorem~1.6]{shahar_gafa_ln}. 
The dual norm of the $\ell_1^d$-norm is $1$-unconditional  with respect to the canonical basis of $\R^d$ \cite[Definition~1.4]{shahar_gafa_ln}. Therefore, \cite[Theorem~1.6]{shahar_gafa_ln} applies under the following assumption.

\begin{Assumption}\label{ass:shahar_theo16}
There exist constants $q_0>2$, $C_0$ and $L$ such that $\zeta\in L^{q_0}$, $X$ is isotropic ($\E \inr{X, t}^2 = \norm{t}_2^2$ for every $t\in\R^d$) and its coordinates have $C_0 \log d$ subgaussian moments: for every $1\leq j\leq d$ and every $1\leq p\leq C_0 \log d$, $\norm{\inr{X,e_j}}_{L^p}\le L\sqrt{p}\norm{\inr{X,e_j}}_{L^2}$.
\end{Assumption}
%
%The unit ball $B(0,1)$ plays a central role in our analysis. In this example, it is isometric to $B_1^d$. 

\noindent
Under Assumption~\ref{ass:shahar_theo16}, if $\sigma=\norm{\zeta}_{L^{q_0}}$, \cite[Theorem~1.6]{shahar_gafa_ln} shows that, for every $\rho>0$, 
\begin{gather*}
\E \sup_{v\in \rho B_1^d \cap r B_2^d} \left|\sum_{i\in[N]}\eps_i \inr{v, X_i}\right|\leq c_2\sqrt{N}\ell^*(\rho B_1^d \cap r B_2^d)\enspace,\\
\E \sup_{v\in \rho B_1^d \cap r B_2^d} \left|\sum_{i\in[N]} \eps_i \zeta_i \inr{v, X_i}\right|\leq c_2\sigma\sqrt{N}\ell^*(\rho B_1^d \cap r B_2^d)\enspace.
\end{gather*}
Local Gaussian mean widths $\ell^*(\rho B_1^d \cap r B_2^d)$ are bounded from above in \cite[ Lemma~5.3]{LM_reg_comp} and computations of $r_M$ and $r_Q$ follow 
\begin{align*}
r_M^2(\rho) &\lesssim_{L,q_0, \gamma_M} 
\begin{cases}
\sigma^2\frac{ d}{N} & \mbox{ if } \rho^2 N \geq \sigma^2 d^2\\
 \rho\sigma\sqrt{\frac{1}{N}\log\Big(\frac{e\sigma d}{\rho\sqrt{N}}\Big)} & \mbox{ otherwise}
\end{cases}
\enspace,\\
\notag r_Q^2(\rho) &
\begin{cases}
 = 0 & \mbox{ if }  N \gtrsim_{L, \gamma_Q}  d \\
 \lesssim_{L, \gamma_Q} \frac{\rho^2}{N}\log\Big(\frac{c(L, \gamma_Q)d}{N}\Big) & \mbox{ otherwise}
\end{cases}
\enspace.
\end{align*}
Therefore, a link function is explicitly given by 
\begin{equation}\label{eq:r_function_LASSO}
r^2(\rho) \sim_{L,q_0, \gamma_Q, \gamma_M} 
\begin{cases}
\max\left(\rho \sigma\sqrt{\frac{1}{N}\log\Big(\frac{e \sigma d}{\rho\sqrt{N}}\Big)}, \frac{\sigma^2 d}{N}\right) & \mbox{ if } N \gtrsim_L d\\
\max\left(\rho \sigma \sqrt{\frac{1}{N}\log\Big(\frac{e \sigma d}{\rho\sqrt{N}}\Big)},\frac{\rho^2}{N}\log\Big(\frac{d}{N}\Big)\right)  & \mbox{ otherwise} 
\end{cases}
\enspace.
\end{equation}

\subsection{The estimators}
\label{sec:estimators}
Let $(T_{K,\lambda}(g, f))_{f, g\in F}$ denote the family of tests defined in \eqref{eq:beat_on_Bk}.
%, they are regularized MOM tests. 
For every function $f\in F$, let $\mathcal B_{K,\lambda}(f)=\{g\in F : T_{K,\lambda}(g,f)\ge 0\}$ denote the set of all functions $g\in F$ that beats $f$. As explained in Section~\ref{Sec:LFT}, these sets will be measured by two metrics.
%the radii of $\mathcal B_{K,\lambda}(f)$ w.r.t. the regularization norm has to be evaluated to deduce a first estimator for which an estimation result of $f^*$ w.r.t. the regularization will be derived. We also want estimation results w.r.t. 
%we estimate the $L^2_P$-metric
%%, which is an unknown metric that needs to be estimated upstream, for instance, 
%via a MOM approach.
% (given that the dataset is corrupted by outliers and that we don't have enough concentration property to estimate it via the classical empirical $L^2_{P_N}$ metric). To that end we also need to introduce another (empirical) distance.
%
%More precisely, two ways to measure the size of $\cB_{K,\lambda}(f)$ for $f\in F$ and associated estimator are now introduced.  
First, let
\begin{gather*}
R_{K,\lambda}^{\text{reg}}(f)=\sup_{g\in \mathcal B_{K,\lambda}(f)}\left\{\|g-f\|\right\} \mbox{ and } \hat f_{K,\lambda}^{(1)}\in \arg\min_{f\in F} R_{K,\lambda}^{reg}(f)\enspace.
\end{gather*}
Next, let
\begin{gather*}
 R_{K,\lambda}^{(2)}(f)=\sup_{g\in \mathcal B_{K,\lambda}(f)}\left\{\MOM{K}{|g-f|}\right\}.
\end{gather*}
%Recall that $\hat f_{K,\lambda}^{(1)}$ is also the solution to
%\begin{equation*}
%\hat f_{K,\lambda}^{(1)}\in \arg\min_{f\in F}C^{(1)}(f) \mbox{ where } C^{(1)}(f)=\min\left\{\rho\ge 0 : R_{K,\lambda}^{\text{reg}}(f)\le \rho\right\}.
%\end{equation*}
%The risk of this estimator will be bounded w.r.t. the regularization norm.
%
%To get estimation results for both the regularization norm and the $L^2_P$-norm, one needs to add another way to measure the size of the sets $\cB_{K,\lambda}(f)$ related to this norm. Given that $P$ is unknown in general, we have to introduce an empirical way to measure the $L^2_P$ diameter of $B_{K,\lambda}(f)$. The classical empirical $L^2_{P_N}$-metric would work under strong assumptions on the design, like a subgaussian property or any other property insuring some isometry between $L^2_{P_N}$ and $L^2_P$ metrics above some level (usually given by $r_Q(\cdot)$, cf. \cite{LM13}). We want to relax also this assumption and consider therefore another ``empirical norm''. We define a second radius of $\cB_{K,\lambda}(f)$ by 
Lemma~\ref{lem:Isometry} below proves that, with large probability, $\MOM{K}{|f-g|}$ and $\norm{f-g}_{L^2_P}$ are isomorphic distances. The second criterion is then given by 
\[
C^{(2)}_{K,\lambda}(f)=\inf\left\{\rho\ge 0 : R_{K,\lambda}^{\text{reg}}(f)\le \rho \text{ and }R_{K,\lambda}^{(2)}(f)\le 85\param{r} r(\rho)\right\}\enspace,
\]
where $r(\cdot)$ is a link function as defined in Definition~\ref{def:the-three-paraemters}. That is a continuous and non-decreasing function such that for all $\rho>0$, $r(\rho)\geq  \max(r_M(\rho,\gamma_M), r_Q(\rho,\gamma_Q))$ where the choice of $\gamma_Q$ and $\gamma_M$ is given in Theorem~\ref{theo:basic-combining-loss-and-reg} below.  
The associated estimator is then given by
\[
\hat f_{K,\lambda}^{(2)}\in \argmin_{f\in F}C^{(2)}_{K,\lambda}(f)\enspace.
\]
%This estimator can be used to obtain estimation results of $\bayes$ for both the regularization norm and the $L^2_P$-norm.
%\textcolor{red}{Verifier si on besoin de cette partie\\
%Finally, we want to design an estimator satisfying simultaneously a sharp oracle inequality and optimal risk bounds for both the regularization and $L^2_P$ norms. We introduce a third pseudo-distance :
%\begin{equation*}
%\forall f,g\in F,\qquad d_K(g, f)=\text{MOM}_K[(Y-f)(f-g)]\enspace,
%\end{equation*}
%the radius of $\cB_{K,\lambda}(f)$ w.r.t. $d_K$ is denoted by $R_{K,\lambda}^{d_K}(f)=\sup_{g\in \mathcal B_{K,\lambda}(f)}\left\{d_K(g,f)\right\}$.
%%
%The third criterion is given by
%%
%\begin{equation*}
%C^{(3)}(f)=\min\left\{\rho\ge 0 : R_{K,\lambda}^{\text{reg}}(f)\le \rho,\;R_{K,\lambda}^{L_{2,N}}(f)\le  \eta r(\rho), R_{K, \lambda}^{d_K}(f)\leq \gamma r^2(\rho)\right\}
%\end{equation*}
%where $\eta$ and $\gamma$ are specified in Theorem~\ref{theo:basic-combining-loss-and-reg} below and the associated estimator is $
%\hat f_{K,\lambda}^{(3)}\in \arg\min_{f\in F}C^{(3)}(f)$.}
%This estimator can be used to obtain estimation results of $\bayes$ for both the regularization norm and the $L^2_P$-norm as well as sharp oracle inequalities. }

\subsection{The sparsity equation} % (fold)
\label{sub:the_sparsity_equation}
By \eqref{GenRiskBound}, estimation rates for $\hat f_{K,\lambda}^{(2)}$ will be derived from upper bounds on $C_{K,\lambda}^{(2)}(\bayes)$.
% (similarly, estimation rates for $\hat f_{K,\lambda}^{(1)}$ follows from a control on $C_{K,\lambda}^{(1)}(\bayes)$). 
To get these, our strategy is to show that $T_{K,\lambda}(\bayes,f)> 0$ for all $f$ such that $\|f-\bayes\|$ or $\|f-\bayes\|_{L^2_P}$ is large. 

Recall that the quadratic / multiplier decomposition of the excess quadratic risk:
\begin{equation}\label{eq:SE1}
T_{K,\lambda}(\bayes,f)=\text{MOM}_K[(f-\bayes)^2-2\zeta(f-\bayes)]+\lambda(\|f\|-\|\bayes\|)\enspace. 
\end{equation}
Let $f\in F$ and $\rho=\norm{f-\bayes}$. When $\rho$ is large and $\norm{f-f^*}_{L^2_P}$ is small, $T_{K,\lambda}(\bayes,f)> 0$ thanks to the regularization term $\lambda(\|f\|-\|\bayes\|)$ in \eqref{eq:SE1} because the quadratic term $(f-\bayes)^2$ is likely to be small.
% and therefore of little help to upper bound the oscillations of the multiplier term $2\zeta(f-\bayes)$.  
% Therefore, only a lower bound on the regularization term $\lambda(\|f\|-\|\bayes\|)$ can be used to prove that $T_{K,\lambda}(\bayes,f)> 0$ in this case. 
 We will therefore derive a  lower bound on the regularization term when the subdifferential of $\norm{\cdot}$ is ``large'' in the following sense.

First, we recall that the subdifferential of $\norm{\cdot}$ in $f\in F$ is the set
\begin{equation*}
(\partial\norm{\cdot})_f = \{z^*\in E^*: \norm{f+h}\geq \norm{f}+z^*(h) \mbox{ for every } h\in E  \}\enspace,
\end{equation*} 
where $(E^*, \norm{\cdot}^*)$ is the dual normed space of $(E,\norm{\cdot})$ (and $E$ is the linear space containing $F$ onto which $\norm{\cdot}$ is defined).
For all $\rho>0$, let $H_\rho$ denote the set
\[
H_{\rho} = \{f \in F : \norm{f-\bayes}=\rho, \; \|f-\bayes\|_{L^2_P} \leq r(\rho)\}
\]where $r(\cdot)$ is the \textit{link function} from Definition~\ref{def:the-three-paraemters}. Let $\Gamma_{\bayes}(\rho)$ denote the union of all subdifferentials of $\norm{\cdot}$ at functions ``close" to $\bayes$ 
\[
\Gamma_{\bayes}(\rho)=\bigcup_{f\in B(\bayes, \rho/20)}(\partial \norm{\cdot})_f\enspace. 
\]

Intuitively, every norm is associated with a notion of ``sparsity'' if one agrees to say that a non-zero function $f^{**}$ is \textit{sparse} w.r.t. the norm $\norm{\cdot}$ when the subdifferential of this norm at $f^{**}$ is a ``large subset'' of the dual sphere (i.e. the sphere of $(E^*, \norm{\cdot}^*)$). Sparse functions $f^{**}$ are useful in our context because a large lower bound on $\norm{f}-\norm{f^{**}}$ (and so for $\norm{f}-\norm{f^{**}}$ when $\norm{f^{**}-f^*}$ is small enough) can be derived when the vector $f-f^{**}$ is in the right direction. This intuition are formalized in the sparsity equation. More precisely, let
\begin{equation*}
\forall \rho>0,\qquad \Delta(\rho) = \inf_{f \in H_{\rho}} \sup_{z^* \in \Gamma_{\bayes}(\rho)} z^*(f-\bayes)\enspace.
\end{equation*}
$\Delta(\rho)$ is a uniform lower bound on $\|f\|-\|f^{**}\|$ if $f^{**}\in B(\bayes, \rho/20)$. Thus, $\|f\|-\|\bayes\|\gtrsim \rho$, if $\sup_{f^{**}\in \Gamma_{\bayes}(\rho)}(\|f\|-\|f^{**}\|)\gtrsim \rho$ or if the following \emph{sparsity equation} of \cite{LM_reg_comp} holds.

\begin{Definition}\label{def:sparsity_equation}
A radius $\rho>0$ satisfies the \textbf{sparsity equation} if $\Delta(\rho) \geq 4\rho/5$.
\end{Definition}

\noindent
If $\rho^*$ satisfies the sparsity equation, so do all $ \rho\geq \rho^*$. Therefore, one can define
\begin{equation}\label{eq:SE}
\rho^* = \inf\left(\rho>0: \Delta(\rho) \geq \frac{4\rho}{5}\right).
\end{equation}

\paragraph*{The sparsity equation in $\ell_1^d$-regularization}
The equation has been solved in this example in \cite[Lemma~4.2]{LM_reg_comp}, 
%Therefore, if the oracle $t^*$ is  close to a sparse vector, the subdifferential in this neighborhood of $\bayes$ (for $\bayes(\cdot)=\inr{\cdot, t^*}$) is large when $\rho$ is and the size of $\Gamma_{\bayes}(\rho)$ is large also. 
recall this result.

\begin{Lemma} \label{lem:sparse_equa_LASSO}
If there exists $v\in\R^d$ such that $v \in t^*+(\rho/20)B_1^d$ and $|{\rm supp}(v)| \leq c\rho^2/r^2(\rho)$ then 
\begin{equation*}
\Delta(\rho)=\inf_{h \in \rho S_1^{d-1}\cap r(\rho)B_2^{d}} \sup_{g \in \Gamma_{t^*}(\rho)} \inr{h, g-t^*}\geq \frac{4\rho}{5}\enspace.
\end{equation*}
where $S_{1}^{d-1}$ is the unit sphere of the $\ell_1^d$-norm and $B_2^{d}$ is the unit Euclidean ball in $\R^d$.
\end{Lemma}

\noindent
If $N\gtrsim s \log(ed/s)$ and if there exists a $s$-sparse vector in $t^*+(\rho/20)B_1^d$, Lemma~\ref{lem:sparse_equa_LASSO} and the choice of $r(\cdot)$ in \eqref{eq:r_function_LASSO} imply that for $\sigma = \norm{\zeta}_{L^{q_0}}$,
\begin{equation*}
\rho^* \sim_{L, q_0} \sigma s \sqrt{\frac{1}{N}\log\left(\frac{ed}{s}\right)} \mbox{ and } r^2(\rho^*)\sim \frac{\sigma^2 s}{N}\log\left(\frac{ed}{s}\right)
\end{equation*}then $\rho^*$ satisfies the sparsity equation and $r^2(\rho^*)$ is the rate of convergence of the LASSO (cf. \cite{LM_reg_comp}).

\section{Main results}\label{sub:main_result} % (fold)
\subsection{Performances of the estimators}

\noindent
Theorem \ref{theo:basic-combining-loss-and-reg} gathers estimation error bounds satisfied by the estimators $\hf_{K,\lambda}^{(j)}$ for $j=1, 2$ defined in Section~\ref{sec:estimators}.

\begin{Theorem}\label{theo:basic-combining-loss-and-reg} Grant Assumptions \ref{ass:Mom2F},~\ref{ass:margin} and \ref{ass:small-ball} and let $r_Q$, $r_M$ anr $r$ denote the functions introduced in Definition~\ref{def:the-three-paraemters} for 
\begin{equation*}
\gamma_Q = \min\left(\frac{1}{661\theta_0}, \frac{1}{1764\theta_r}\right), \gamma_M = \frac{\eps}{ 168} \mbox{ and } \eps = \frac{3}{331\theta_0^2}.
\end{equation*} Let $\rho^*$ be defined in \eqref{eq:SE} and let $K^*$ denote the smallest integer such that 
\[
K^*\ge \max\left( \frac{8K_o}{7}, \frac{N \eps^2 r^2(\rho^*)}{336\theta_m^2}\right)\enspace.
\] 
%For any $K\ge 1$ define $\rho_K$ as the solution of $r(\rho_K)=\cabs{5}\param{0}^2\param{m}\sqrt{\frac K{N}}$. 
For all $K\ge 1$, let $\rho_K$ be a solution of  $r^2(\rho_K)= [16 \theta_m^2/(\eps^2\alpha)]\sqrt{K/N}$. Assume that for every $i\in \cI$, $K\in[K^*, N]$ and $f\in F\cap B(\bayes,\rho_K)$, 
 \begin{equation}\label{eq:robust_theo_basic}
 2(P_i-P) \zeta (f-\bayes)\leq \eps\max\left(\frac{16 \theta_m^2}{\eps^2\alpha}\frac{K}{N}, r^2_M(\rho_K, \gamma_M), \norm{f-\bayes}^2_{L^2_P}\right)\enspace.
 \end{equation}
%Let $K\in[c r^2(\rho^*)N,N]$ and $\lambda$ be such that
%\begin{equation}\label{eq:choice_reg_param}
%\frac{10}{14}\left(\frac{\delta \beta^2\alpha^2 }{2}+4\alpha_c\right)\frac{r^2(\rho_K)}{\rho_K} < \lambda < \frac{10}{9}\left(\frac{\delta \beta^2\alpha^2}2-2\alpha_c\right)\frac{r^2(\rho_K)}{\rho_K}.
%\end{equation}
%Then with probability larger than  $1-3\exp\left(-\delta^2\min(8\delta^2\theta^2 (\alpha\beta u)^4 N/(9\tau^4),2K)\right)$,
For all $K\in[K^*,N/(84\param{r}^2\theta_0^2)]$, on an event $\Omega_1(K)$ such that $\bP(\Omega_1(K))\ge 1-4\exp(-K/1008)$, 
the estimators $\hf_{K,\lambda}^{(j)}$ for $j=1, 2$ defined in Section~\ref{sec:estimators} satisfy
\begin{equation*}
 \norm{\est^{(1)}-\bayes}\leq \rho_K\enspace,
\end{equation*}and 
\begin{equation*}
 \norm{\est^{(2)}-\bayes}\leq \rho_K,\quad\norm{\est^{(2)}-\bayes}_{L^2_P}\leq 340 \theta_0\theta_r r(\rho_K)
\end{equation*}when the regularization parameter satisfy
\[
\frac{20\eps}{7}\frac{r^2(\rho_K)}{\rho_K}<\lambda<\frac{10}{331\theta_0^2}\frac{r^2(\rho_K)}{\rho_K}\enspace.
\]
\end{Theorem}

To the best of our knowledge, Theorem~\ref{theo:basic-combining-loss-and-reg} provides the first statistical performance of an estimator operating in such a ``nasty'' environment: the dataset may be corrupted by complete outliers,  the informative data may be heavy-tailed and their distribution $P_i$ for $i\in\cI$ is only asked to have a $L^2$ and $L^1$ geometry over $F-f^*$ equivalent to that of $P$. 
%In particular, the distribution $P_i, i\in\cI$ of the informative data may be very different from the one associated to the oracle $P$, but still Theorem~\ref{theo:basic-combining-loss-and-reg} shows that this is enough for the estimation of $f^*$ and this happens even if those ``little informative data'' are corrupted with very bad outliers. 
The most surprising thing is that the rate we obtain for $K=K^*$ in Theorem~\ref{theo:basic-combining-loss-and-reg}, i.e. $r(\rho_{K^*})$ when the number of outliers $K_o$ is less than $N r^2(\rho^*)$ is the minimax rate we would have gotten in a very good i.i.d. subgaussian framework with independent noise.
% (more details are provided later on this remark). 
 This means that the quality of a dataset does not have to be as good as it is classically assumed in the literature to make estimation possible: all we need is that a large fraction of the data should be independent (even though we believe that some ``weak dependence'' could also be introduced) and distributed according to distributions inducing  $L^1$ and $L^2$ geometries equivalent to the $L^2_P$ one. 
% That is the only thing that is really needed to make estimation at all possible; that is pretty weak given that $f^*$ is defined as the projection onto $F$ of $Y$ in $L^2_P$ and that the risk is the $L^2_P$ estimation risk.  
%Furthermore, these informative data are not necessarily i.i.d. $P$, it is only assume that the $L^2_P$ and $L^2_{P_i}$ are close. 

In Theorem~\ref{theo:basic-combining-loss-and-reg}, $K$ can be as small as the infimum between the number of outliers and $N$ times the minimax rate of convergence. Henceforth, if the optimal rate is known, as in \cite{LugosiMendelson2016}, Theorem~\ref{theo:basic-combining-loss-and-reg} shows that Le Cam's champion of the median of means tournament with $K=K^*$ reaches the same performances as any champion in this paper.
%, even if there are some outliers in the data-set and informative data are not exactly i.i.d.
 Theorem~\ref{theo:basic-combining-loss-and-reg} is thus an extension of \cite{LugosiMendelson2016} to a non-i.d. corrupted setting for Le Cam's champion. Moreover, our control improves theirs if the upper bound on the radius of $\bayes$ used in \cite{LugosiMendelson2016} is pessimistic (cf. Example~\ref{Ex:LugMen} in Section~\ref{sec:Example}).
%
%Recall that we have seen in Section~\ref{sec:Example} that the $\hat f_{K,\lambda}$ are champions of a MOM's tournament in the sense of Lugosi and Mendelson \cite{LugosiMendelson2016}. In this sense, their result is stronger than ours since they prove that \emph{any} champion has a properly controlled risk if $K$ is adequately chosen. On the other hand, the main advantage of our approach is that the \emph{definition} of Birgé's estimator does not require the computation of a theoretical upper bound on the radius of the oracle.  

%Let us comment on the assumptions. 
%One cannot hope to get any result if the probability distributions $\{P_i\}_{i\in \cI}$ are not related to the probability distribution of the ``new'' data $(X,Y)$, which is $P$.  
Assumption~\ref{ass:Mom2F} is automatically satisfied in the i.i.d. case and so is Assumption~\eqref{eq:robust_theo_basic}. 
%In the i.i.d. setup one just may forget about those two assumptions. 
%
Theorem~\ref{theo:basic-combining-loss-and-reg} goes beyond this i.i.d. setup, relaxing 
%with standard strong subgaussian assumptions.
% . 
the i.d. assumptions into proximity assumptions between $L_{P_i}^2$ and $L^2_P$ geometries, for informative data. 
%Assumption~\ref{ass:Mom2F} simply requires that the distributions $P_i$ and $P$ have equivalent $L^2$ moment for a proportional number of  $i\in[N]$. This is a weak assumption given that we want to estimate the oracle $\bayes$ (w.r.t. $P$) in $L^2_P$. Note also that Assumption~\ref{ass:Mom2F} does not exclude that a proportion of the data may have $L^2$ moments not related to the one of $P$ for some or every functions in $F$. 
%
% More precisely, Assumption~\eqref{eq:robust_theo_basic} ensures that the best $L^2_{P_i}$-approximation of $Y_i$ in $F$ is close to the best $L^2_P$-approximation $\bayes$. 
% These weak assumptions are sufficient for MOM estimators to recover the optimal statistical properties of (penalized) ERM in an i.i.d. setup, with subgaussian class of functions and subgaussian noise $\zeta$ independent of $X$.

\paragraph*{Risk bounds in $\ell_1^d$ regularization}Let us now 
%give an explicit application of Theorem~\ref{theo:basic-combining-loss-and-reg} in the ``LASSO case''.
%First, we need to 
compute explicit values of $\rho_K$ and $\lambda\sim r^2(\rho_K)/\rho_K$ in the $\ell_1^d$-regularization case. Let $K\in[N]$ and $\sigma = \norm{\zeta}_{L^{q_0}}$. The equation $K=c r(\rho_K)^2N$ is solved by
\begin{equation}\label{eq:LASSO_choice_rho_K}
 \rho_K\sim_{L,q_0} \frac{K}{\sigma} \sqrt{\frac{1}{N}\log^{-1}\left(\frac{\sigma^2 d}{K}\right)}
 \end{equation} for the $r(\cdot)$ function defined in \eqref{eq:r_function_LASSO}.
 Therefore,  \begin{equation}\label{eq:reg_param_lasso}
 \lambda \sim \frac{r^2(\rho_K)}{\rho_K} \sim_{L,q_0} \sigma \sqrt{\frac{1}{N}\log\left(\frac{e\sigma d}{\rho_K\sqrt{N}}\right)} \sim_{L,q_0} \sigma \sqrt{\frac{1}{N}\log\left(\frac{e\sigma^2 d}{K}\right)}\enspace.
 \end{equation}
The regularization parameter depends on the ``level of noise'' $\sigma$, the $L^{q_0}$-norm of $\zeta$. This parameter is unknown in practice. Nevertheless, it can be estimated and replaced by this estimator in the regularization parameter as in \cite[Sections~5.4 and 5.6.2]{MR3307991}. 

\subsection{Adaptive choice of $K$ by Lepski's method}
The main drawback of Theorem~\ref{theo:basic-combining-loss-and-reg} is that optimal rates are only achieved when $K\approx K^*$. Since $K^*$ is unknown, it cannot be used in general. This issue is tackled in this section by Lepski's method. 

Let $K_1=K^*$ and $K_2=N/(84\theta_0^2\param{r}^2)$ be defined as in Theorem~\ref{theo:basic-combining-loss-and-reg}. For any integer $K\in[K_1,K_2]$, let $\rho_K$ and $\lambda$ be defined as in Theorem~\ref{theo:basic-combining-loss-and-reg} and for $j=1, 2$ denote by $\hat f_K^{(j)} =\est^{(j)}$ for this choice of $\lambda$. % In particular, it follows from Theorem~\ref{theo:basic-combining-loss-and-reg}, that, with probability at least  $1-3e^{-\frac{(\delta\alpha\beta u)^3\theta^2}8 K}$, 
% \begin{equation*}
% \norm{\esti-\bayes}\leq \rho_K \an \norm{\esti-\bayes}_{L_2(P)}\leq \frac{\tau^2}{\delta(\alpha\beta)^2u}r(\rho_K).
% \end{equation*}
These estimators are the building blocks of the following confidence sets.
For all $f\in F$, let 
\begin{equation*}
\hat B^{(2)}_K(f)=\left\{g\in F : \MOM{K}{|g-f|}\le 28900\theta_r^2\theta_0 r(\rho_K)\right\}\enspace.
\end{equation*}
%\textcolor{red}{ and the one associated to $d_K$ is defined by  
% \begin{equation*}
% \hat B^{d_K}_K(f)=\{g\in F : d_K(g, f)\geq - \nu r^2(\rho_K)\},
% \end{equation*}
%where $\mu$ and $\nu$ will be set in Theorem~\ref{theo:main_baby_Lepski}.
%The notion of proximity w.r.t. $d_K$ used here to define $\hat B^{d_K}_K(f)$ is somehow the opposite of the one used to criteria $C^{(3)}(\cdot)$ and the associated estimator $\esti^{(3)}$.}
 Now, let
% define for every $K\in [N]$, three confidence sets respectively associated to the three estimators $\esti^{(j)}, j=1, 2, 3$:
\begin{equation*}
R_K^{(1)}=B(\esti^{(1)},\rho_K),\quad R_K^{(2)}=B(\esti^{(2)},\rho_K)\cap \hat B^{(2)}_K(\hat f_K^{(2)}) 
%\textcolor{red}{\mbox{ and }R_K^{(3)}=B(\esti^{(3)},\rho_K)\cap \hat B^{L_{2,N}}_K(\hat f_K^{(3)})\cap \hat B^{d_K}_K(\esti^{(3)})}
\end{equation*}
and for every $j=1, 2$, let
\[
\hat{K}^{(j)}=\inf\left\{K\in[K_2]: \bigcap_{J=K}^{K_2}R_J^{(j)}\ne \emptyset\right\}\enspace.
\]

% \bluenote{GUILLAUME: This is a minor detail but I would rather take 
% \begin{equation*}
% \hat{K}=\inf\left\{K\in\{1,\ldots,N\}: \bigcap_{J=K}^{N}R_J\ne \emptyset\right\}
% \end{equation*}that is up to $N$ instead of $u^2 N$. The proof is identical for both estimators.}

% \textcolor{red}{MATTHIEU : No problem for me. BE AWARE THOUGH THAT THE KNOWLEDGE OF THE ROBUSTNESS PARAMETERS $\theta$ and $\tau$ IS NECESSARY TO BUILD $\hat B_K$}

Finally, define adaptive (to $K$) estimators via Lepski's method: for $j=1, 2$, $
\hf_{LE}^{(j)} \in \bigcap_{J=\hat{K}^{(j)}}^{K_2}R_J^{(j)}$.

\begin{Theorem}\label{theo:main_baby_Lepski}
Grant assumptions and notations of Theorem~\ref{theo:basic-combining-loss-and-reg}. There exist absolute constants $(\cabs{i})_{1\le i\le 2}$ such that the estimators $\hf_{LE}^{(j)}$  for $j=1, 2$ 
%\textcolor{red}{$\nu = \frac{3c_3}{\sqrt{c_1}}\beta\tau$} 
satisfy for  every $K\in[ K^*, N/(84\theta_0^2\param{r}^2)]$, with probability at least $1-\cabs{1}\exp\left(-\cabs{2}K\right)$, 
\begin{equation*}
\norm{\hf_{LE}^{(1)}-\bayes}\leq 2\rho_{K}\enspace,
\end{equation*}and
 \begin{equation*}
  \norm{\hf_{LE}^{(2)}-\bayes}\leq 2\rho_{K},\quad\norm{\hf_{LE}^{(2)}-\bayes}_{L^2_P}\leq 680 \theta_r\theta_0 r(2\rho_K)\enspace.
 \end{equation*}
In particular, for $K=K^*$,  if the following regularity assumption holds: there exists an absolute constant $c_3$ such that for all $\rho>0$, $r(2\rho)\leq c_3 r(\rho)$ then with probability at least  
\[
1-c_1\exp\left(-c_4 N \max\left(\frac{K_o}{N}, \frac{r^2(\rho^*)}{\theta_0^4\theta_m^2}\right)\right)
\]
then, 
\begin{equation*}
\norm{\hf_{LE}^{(2)}-\bayes}_{L^2_P}\leq c_5 \max\left(\theta_0^4\theta_m^2\frac{K_o}{N}, r^2(\rho^*)\right)\enspace.
%\an R(\hf_{LE}^{(2)})\leq R(\bayes) + c_4 r_*^2.
\end{equation*}
\end{Theorem}
% In other words, $\hf_{BL}$ behaves (up to constants) as well as the estimator $\hf_{K^{\text{opt}}}$.
Recall an optimality result from \cite{LM13}. Assume that all $(X_i, Y_i), i\in[N]$ are distributed according to  $(X,Y^{\bayes})$, where $\bayes\in F$, $Y^{\bayes}=\bayes(X) + \zeta$ and $\zeta$ is a centered Gaussian variable with variance $\sigma$ independent of $X$. Assume that $F$ is $L$-subgaussian : for every $f\in F$ and $p\ge 2$, $\norm{f}_{L^p}\leq L\sqrt{p} \norm{f}_{L^2}$. Then, \cite[Theorem~A${}^{\prime}$]{LM13} proves that if $\tilde f_N$ is an estimator such that for every $f^*\in F$ and every $r>0$, with probability at least $1-c_0\exp(- \sigma^{-1}r^2 N/c_0)$, $\norm{\tilde f_N-\bayes}_{L^2_P}\leq \zeta_N$, then necessarily
\begin{equation}\label{eq:minimax_rate_LM13}
\zeta_N\gtrsim \min\left(r, {\rm diam}(F, L^2_P)\right).
\end{equation}
When $Y^{\bayes}=\bayes(X) + \zeta$, $c\sim 1/\param{m}\sim 1/\sigma$. Applying this result to $r=r(\rho_K)$ for some given $K\geq K^*$ shows no procedure can estimate $\bayes$ in $L^2_P$ uniformly over $F$ with confidence at least $1-c_0\exp(-K/c_0)$ at a rate better than $r(\rho_K)$ (we implicitly assumed that $r(\rho_K)\leq {\rm diam}(F, L^2_P)$ since $r(\rho_K)$ can obviously be replaced by $r(\rho_K)\wedge {\rm diam}(F, L^2_P)$ in all results). Moreover, this rate is minimax since \cite[Theorem~A]{LM13} also shows that the ERM over $\rho_K B$, $\hat f^{ERM}_N \in\argmin_{f\in \rho_{K}B} P_N\ell_f$, satisfies $\norm{\hat f^{ERM}_N - \bayes}_{L^2}\lesssim r(\rho_K)$ with probability at least $1-c_0\exp(-\sigma^{-1} r^2(\rho_{K})N/c_0)$ when $\sigma\gtrsim r_Q(\rho_{K})$. 
%This proves that, when the noise level is non trivial (that is $\sigma\gtrsim r_Q(\rho_{K})$), the ERM over $\rho_{K}B$ is minimax in the exponentially high confidence regime. 

Theorem~\ref{theo:main_baby_Lepski} shows that $\hat f_{LE}$ achieves the same rate of convergence with the same exponentially high confidence as a minimax estimator does in the Gaussian regression model (with independent noise). 
%But the main difference here is that in Theorem~\ref{theo:basic-combining-loss-and-reg}, t
These rates are achieved here under very weak stochastic assumptions allowing the presence of outliers, without assuming that the regression function lies in $F$ or that the data are i.i.d.. Compared to \cite{LugosiMendelson2016}, using a Lepski method, we don't have to \emph{choose} the integer $K$ in advance, we let the data decide the best choice and automatically get an estimator with the correct minimax rate of convergence. Moreover, the regularization parameter is chosen adaptively, which yields to exact minimax rates and, since this minimax rate is not required to build the estimators, these are naturally adaptive. 

\paragraph*{Adaptive results in $\ell_1^d$ regularization}
The following result follows from Theorem~\ref{theo:main_baby_Lepski} together with the computation of $\rho^*$, $r_Q$, $r_M$ and $r$ from the previous sections. This is a slight extension of Theorem~\ref{theo:mom_lasso} to the case where the oracle $t^*$ is not exactly sparse but close to a sparse vector.

\begin{Theorem}\label{theo:mom_lasso_sharp}
Assume that $X$ is isotropic and  
\begin{enumerate}
	\item[o)] there exist $s\in[N]$ such that $N\geq c_1 s \log(ed/s)$ and $v\in\R^d$ such that $\norm{t^*-v}_1\leq \sigma s \sqrt{\log\left(ed/s\right)/N}/20$ and $|{\rm supp}(v)|\leq s$.  
	\item[i')] $|\cI|\geq N/2$ and $|\cO|\leq c_1 s \log(ed/s)$,
	\item[ii)] $ \zeta=Y-\inr{X, t^*} \in L_{q_0}$ for some $q_0>2$
	\item[iii')] for every $1\leq p\leq C_0 \log(ed)$, $\norm{\inr{X,e_j}}_{L_p}\leq L \sqrt{p}\norm{\inr{X,e_j}}_{L_2}$ where $(e_j)_{j\in[d]}$ is the canonical basis of $\R^d$ and $C_0$ is some absolute constant,
	\item [iv')] there exists $\theta_0$ such that  $\norm{\inr{X, t}}_{L^1}\leq \theta_0\norm{\inr{X, t}}_{L^2}$, for all $t\in \R^d$,
	\item [v)] there exists $\theta_m$ such that ${\rm var}(\zeta\inr{X, t})\leq \param{m}^2 \norm{t}_2^2$, for all $t\in\R^d$.
\end{enumerate}
The MOM-LASSO estimator $\hat{t}_{LE}$ such that $\hat f_{LE}=\inr{\hat{t}_{LE}, \cdot}$ satisfies, with probability at least $1-c_2 \exp(-c_3 s \log(ed/s))$,  for every $1\leq p\leq 2$, 
\begin{equation*}
\norm{\hat{t}_{LE} - t^*}_p\leq c_4(L, \theta_m)\norm{\zeta}_{L_{q_0}}s^{1/p}\sqrt{\frac{1}{N}\log\left(\frac{ed}{s}\right)}\enspace,
\end{equation*}
\end{Theorem}
In particular, Theorem~\ref{theo:mom_lasso_sharp} shows that, for our estimator contrary to the one in \cite{LugosiMendelson2016}, the sparsity parameter $s$ does not have to be known in advance in the LASSO case.
\proof
It follows from Theorem~\ref{theo:main_baby_Lepski}, the computation of $r(\rho_K)$ from \eqref{eq:r_function_LASSO} and $\rho_K$ in \eqref{eq:LASSO_choice_rho_K} that with probability at least $1-c_0\exp(- cr(\rho_K)^2N/\overline{C})$, $\norm{\hat{t}_{LE} - t^*}_1\leq \rho_{K^*}$ and $\norm{\hat t_{LE} - t^*}_2\lesssim r(\rho_K)$.
The result follows since $\rho_{K^*}\sim\rho^* \sim_{L, q_0} \sigma s \sqrt{\frac{1}{N}\log\left(\frac{ed}{s}\right)}$
and $\norm{v}_p\leq \norm{v}_1^{-1+2/p}\norm{v}_2^{2-2/p}$ for all $v\in\R^d$ and $1\leq p\leq2$.

\endproof

%Theorem~\ref{theo:main_lasso} is a sub-Gaussian deviation bound obtained under a limited moment assumption and no independence between the noise and the design. The noise $\zeta$ only needs to have $q_0$ moments for some $q_0>2$, which is almost the minimal assumption one needs to get a $L_2$-estimation result. Moreover, the rate of convergence is the minimax one, that is we really get the $\sqrt{\log(ed/s)/N}$ rate of convergence and not only the ``classical'' $\sqrt{(\log d)/N}$ which is usually found in the literature on the LASSO even under stronger assumption like a statistical model with  subgaussian design and independent gaussian noise. 

%\bluenote{GUILLAUME: The name ``baby Lepski'' is a bit weird. Would it be possible to use another name?}b

\section{Proofs}\label{sec:Proofs}
In all the proof section, we denote by $\bP$ the distribution of $(X_1,\ldots,X_N)$ and $\E$ the corresponding expectation. For any non-empty subset $B\subset [N]$ and any $f:\cX\to\R$ for which it makes sense, let $\overline P_Bf=\frac1{|B|}\sum_{i\in B}P_if$. For any $f\in L^2_P$ and $r\ge 0$, let 
\[
B_2(f,r)=\{g\in L^2_P : \norm{f-g}_{L^2_P}\le r\},\quad S_2(f,r)=\{g\in L^2_P : \norm{f-g}_{L^2_P}= r\}\enspace.
\]We consider the set of indices of blocks $B_k$ containing only informative data:
\begin{equation*}
\cK = \left\{k\in[K]: B_k\subset \cI\right\}.
\end{equation*}

\subsection{Lower Bound on the quadratic process}

\begin{Lemma}\label{lem:UBQP} Grant Assumptions~\ref{ass:Mom2F} and~\ref{ass:small-ball}. Fix $\eta\in (0,1)$, $\rho>0$ and let $\alpha,\gamma_Q,\gamma,x\in (0,1)$ be such that $\gamma\left(1-\alpha-x-32\theta_0 \gamma_Q\right) \ge 1-\eta$. Let $K\in[ K_o/(1-\gamma),N\alpha/(2\theta_0\param{r})^2]$.

There exists an event $\Omega_Q(K, \rho)$ such that $\bP\left(\Omega_Q(K, \rho)\right)\geq 1-\exp(-K\gamma x^2/2)$ on which for all $f\in B(\bayes,\rho)$ if $\norm{f-\bayes}_{L^2_P}\ge r_Q(\rho,\gamma_Q)$ then
\begin{equation*}
\notag Q_{\eta,K}(|f-\bayes|)\ge \frac1{4\theta_0}\norm{f-\bayes}_{L^2_P}\mbox{ and }
 Q_{\eta,K}((f-\bayes)^2)\ge \frac1{(4\theta_0)^2}\norm{f-\bayes}_{L^2_P}^2\enspace.
\end{equation*}
\end{Lemma}
\begin{proof}
For all $f\in F-\{f^*\}$, let $n_f=(f-\bayes)/\norm{f-\bayes}_{L^2_P}$. For $i\in \cI$, $P_i|n_f|\ge \theta_0^{-1}$ by Assumption~\ref{ass:small-ball} and $P_in_f^2\le \param{r}^2$ by Assumption~\ref{ass:Mom2F}. By Markov's inequality, for all $k\in \cK$,
\[
\bP\left(|P_{B_k}|n_f|-\overline P_{B_k}|n_f||>\frac{\param{r}}{\sqrt{\alpha|B_k|}}\right)\le \alpha
\]
and so
\[
\bP\left(P_{B_k}|n_f|\ge \frac1{\theta_0}-\frac{\param{r}}{\sqrt{\alpha|B_k|}}\right)\ge 1-\alpha\enspace.
\]
Since $K\le[\alpha/(2\theta_0\param{r})]^2N$ then $|B_k|=N/K\geq [\alpha/(2\theta_0\param{r})]^2$ and so we have 
\begin{equation}\label{eq:markov_1}
\bP\left(P_{B_k}|n_f|\ge \frac1{2\theta_0}\right)\ge 1-\alpha\enspace.
\end{equation}

Let $\phi$ denote the function defined by $\phi(t)=(t-1)I(1\le t\le 2)+I(t\ge 2)$ for all $t\in\R_+$ and, for all $f\in F-\{f^*\}$, let $Z(f)=\sum_{k\in[K]}I(4\theta_0P_{B_k}|n_f|\ge 1)$. Since $I(t\ge 1)\ge \phi(t)$ for any $t\geq0$ then $Z(f)\ge \sum_{k\in\cK}\phi\left(4\theta_0P_{B_k}|n_f|\right)$. Since $\phi(t)\ge I(t\ge 2)$ for all $t\geq0$, it follows from \eqref{eq:markov_1} that
\begin{align*}
\E\left[\sum_{k\in \cK}\phi\left(4\theta_0P_{B_k}|n_f|\right)\right]\ge \sum_{k\in \cK}\bP\left(4\theta_0P_{B_k}|n_f|\ge 2\right)\ge |\cK|(1-\alpha)\enspace.
\end{align*}
Therefore, for all $f\in F$, we have
\begin{align*}
 Z(f)\ge |\cK|(1-\alpha)+\sum_{k\in \cK}\left(\phi\left(4\theta_0P_{B_k}|n_f|\right)-\E\left[\phi\left(4\theta_0P_{B_k}|n_f|\right)\right]\right)\enspace.
\end{align*}
Let $\cF=\{f\in B(\bayes, \rho) : \norm{f-\bayes}_{L^2_P}\geq r_Q(\rho,\gamma_Q)\}$. By the bounded difference inequality (cf. \cite[Lemma~1.2]{MR1036755} or \cite[Theorem~6.2]{BouLugMass13}, there exists an event $\Omega(x)$ such that $\bP(\Omega(x))\ge 1-\exp(-x^2|\cK|/2)$, on which
\begin{multline*}
\sup_{f\in \cF 
%: \norm{f-\bayes}\ge \sqrt{\frac DN}
} \left|\sum_{k\in \cK}\left(\phi\left(4\theta_0P_{B_k}|n_f|\right)-\E\left[\phi\left(4\theta_0P_{B_k}|n_f|\right)\right]\right)\right|\\
\le \E\sup_{f\in \cF
% : \norm{f-\bayes}\ge \sqrt{\frac DN}
 } \left|\sum_{k\in \cK}\left(\phi\left(4\theta_0P_{B_k}|n_f|\right)-\E\left[\phi\left(4\theta_0P_{B_k}|n_f|\right)\right]\right)\right|+|\cK|x\enspace.
\end{multline*}
By the Gin{\'e}-Zynn symmetrization argument \cite[Lemma~11.4]{BouLugMass13},
\begin{align*}
 \E\sup_{f\in \cF } \left|\sum_{k\in \cK}\left(\phi\left(4\theta_0P_{B_k}|n_f|\right)-\E\left[\phi\left(4\theta_0P_{B_k}|n_f|\right)\right]\right)\right|\le2\E\sup_{f\in \cF } \left|\sum_{k\in \cK}\epsilon_k\phi\left(4\theta_0P_{B_k}|n_f|\right)\right|
\end{align*}where $(\eps_k)_{k\in\cK}$ are independent Rademacher variables independent of the data. Moreover, $\phi$ is 1-Lipschitz and $\phi(0)=0$. By the contraction principle (cf. \cite[Theorem~4.12]{LT:91} or \cite[Theorem 11.6]{BouLugMass13}),

\[
\E\sup_{f\in \cF
% : \norm{f-\bayes}_{L^2}\ge \sqrt{\frac DN}
 } \left|\sum_{k\in \cK}\epsilon_k\phi\left(4\theta_0P_{B_k}|n_f|\right)\right| \le 4\theta_0\E\sup_{f\in \cF
%  : \norm{f-\bayes}_{L^2}\ge \sqrt{\frac DN}
  } \left|\sum_{k\in \cK}\epsilon_kP_{B_k}|n_f|\right| \enspace.
\]
Applying again the symmetrization and contraction principles, we get, 
\[
\E\sup_{f\in \cF} \left|\sum_{k\in \cK}\epsilon_kP_{B_k}|n_f|\right|\le \frac{4K}N\E\sup_{f\in \cF} \left|\sum_{i\in\cup_{k\in \cK}B_k}\epsilon_in_f(X_i)\right|\enspace.
\]
%on $\Omega(x)$,
%\begin{align*}
% Z(f)\ge |\cK|\left(1-\alpha-x-\frac{32\theta_0 K}{|\cK| N}\E\sup_{f\in \cF
%%  : \norm{f-\bayes}_{L^2}\ge \sqrt{\frac DN}
%} \left|\sum_{i\in\cup_{k\in \cK}B_k}\epsilon_in_f(X_i)\right| \right)\enspace.
%\end{align*}
It follows from the convexity of $F$ that for all $f\in \cF$,  $r_{Q}(\rho,\gamma_Q)n_f\in F-\bayes$ and it also belongs to the $L^2_P$ sphere of radius $r_Q(\rho,\gamma_Q)$. Therefore, by definition of $r_Q:=r_Q(\rho,\gamma_Q)$ and for $J = \cup_{k\in \cK}B_k$,
\begin{align*}
\E\sup_{f\in \cF} &\left|\sum_{i\in J}\epsilon_in_f(X_i)\right|=\frac{1}{r_Q}\E\sup_{f\in F\cap S_2(\bayes,r_Q)} \left|\sum_{i\in J}\epsilon_i(f-\bayes)(X_i)\right| \le \gamma_Q\frac{|\cK|N}{K}\enspace. 
\end{align*}

In conclusion, on $\Omega(x)$, all $f\in \cF$ is such that
% such that $\norm{f-\bayes}_{L^2}\ge \sqrt{D/N}$,
\begin{align*}
 Z(f)\ge |\cK|\left(1-\alpha-x-32\theta_0\gamma_Q \right)\ge (1-\eta)K\enspace.
\end{align*}
In other words, on $\Omega(x)$, for all $f\in \cF$, there exist at least $(1-\eta)K$ blocks $B_k$ such that $P_{B_k}|n_f|\ge (4\theta_0)^{-1}$. For any of these blocks $B_k$, $P_{B_k}n_f^2\ge (P_{B_k}|n_f|)^2$, hence, on $\Omega(x)$, $Q_{\eta,K}[|n_f|]\ge (4\theta_0)^{-1}$ and $Q_{\eta,K}[n_f^2]\ge (4\theta_0)^{-2}$. 
\end{proof}

\subsection{Upper Bound on the multiplier process}

\begin{Lemma}\label{lem:proc_multiplicatif} Grant Assumption~\ref{ass:margin}. Fix $\eta\in (0,1)$, $\rho\in(0,+\infty]$, and let $\alpha,\gamma_M,\gamma,x$ and $\eps$ be positive  absolute constants such that $\gamma\left(1-\alpha-x-8 \gamma_M/\eps\right) \ge 1-\eta$. 
Let $K\in[K_o/(1-\gamma),N]$. There exists an event $\Omega_M(K, \rho)$ such that $\bP(\Omega_M(K, \rho))\ge 1-\exp(-\gamma K x^2/2)$ and on $\Omega_M(K, \rho)$, for all $f\in B(\bayes, \rho)$ there is at least $(1-\eta)K$ blocks $B_k$ with $k\in\cK$ such that
\begin{align*}
\left|2(P_{B_k}-\overline P_{B_k})(\zeta(f-\bayes))\right| \leq \eps\max\left(\frac{16 \theta_m^2}{\eps^2\alpha}\frac{K}{N}, r^2_M(\rho, \gamma_M), \norm{f-\bayes}^2_{L^2_P}\right)\enspace.
\end{align*} 
\end{Lemma}

\begin{proof}
For all $k\in [K]$ and $f\in F$, define $W_k(f)=2(P_{B_k}-\overline P_{B_k})\left(\zeta(f-\bayes)\right)$ and 
\begin{equation*}
\gamma_k(f)= \eps\max\left(\frac{16 \theta_m^2}{\eps^2\alpha}\frac{K}{N}, r^2_M(\rho, \gamma_M), \norm{f-\bayes}^2_{L^2_P}\right)\enspace.
\end{equation*}

Let $f\in F$ and $k\in \mathcal K$. It follows from Markov's inequality and Assumption~\ref{ass:margin} that
\begin{align}\label{eq:multi_1}
\nonumber \bP&\left[2\Big|W_k(f)\Big|\ge \gamma_k(f)\right]\leq \frac{4\E \left[\Big(2(P_{B_k}-\overline P_{B_k})(\zeta (f-\bayes))\Big)^2\right]}{ \frac{16\param{m}^2}{\alpha}\norm{f-\bayes}_{L^2_P}^2\frac{K}{N}}\\
& \leq \frac{\alpha \sum_{i\in B_k}{\rm var}_{P_i}(\zeta (f-\bayes))}{|B_k|^2 \param{m}^2\norm{f-\bayes}_{L^2_P}^2\frac{K}N} \leq \frac{\alpha \param{m}^2 \norm{f-\bayes}_{L^2_P}^2}{|B_k|  \param{m}^2\norm{f-\bayes}_{L^2_P}^2\frac{K}N} =\alpha \enspace.
\end{align}
%The last inequality follows from $K\leq c r^2_M(\rho)N$. 
%Hence, $\bP\left[\Big|g_f(W_k)\Big|< \gamma_k(f)/2\right]\ge 1-\alpha$. 
%
Denote $J=\cup_{k\in\mathcal K}B_k$ and remark that $J\in \cJ$ as defined in Definition~\ref{def:the-three-paraemters}. Let $r_M := r_M(\rho,\gamma_M)$ for simplicity. We have
\begin{align*}
 \E&\sup_{f\in B(f^*, \rho)}\sum_{k\in\mathcal K}\epsilon_k\frac{W_k(f)}{\gamma_k(f)}
\leq 2\E \sup_{f\in B(f^*, \rho)}\left|\sum_{k\in\mathcal K}  \frac{\eps_k(P_{B_k}-\overline P_{B_k})(\zeta (f-\bayes))}{\eps\max(r_M^2, \norm{f-\bayes}^2_{L^2_P})}\right|\\
&\leq \frac{2}{\epsilon r_M^2}
\E \left[\sup_{f\in B(f^*, \rho)\setminus B_2(\bayes, r_M)}\left|\sum_{k\in\mathcal K} \eps_k(P_{B_k}-\overline P_{B_k})\left(\zeta r_M\frac{f-\bayes}{\norm{f-\bayes}_{L^2_P}}\right)\right|\right.\\
&\qquad\qquad\qquad\;\left.\vee\sup_{f\in B(f^*, \rho)\cap B_2(\bayes, r_M)}\left|\sum_{k\in\mathcal K} \eps_k(P_{B_k}-\overline P_{B_k})\left(\zeta (f-\bayes)\right)\right|\right]\\
&\leq \frac{2}{\epsilon r_M^2}
\E \sup_{f\in B(f^*, \rho)\cap B_2(\bayes, r_M)}\left|\sum_{k\in\mathcal K} \eps_k(P_{B_k}-\overline P_{B_k})\left(\zeta (f-\bayes)\right)\right|\enspace,
\end{align*}where in the last but one inequality we used that $F$ is convex and the same argument as in the proof of Lemma~\ref{lem:UBQP}. Moreover, since the random variables $((\zeta_i(f-f^*)(X_i) -P_i\zeta (f-f^*)):i\in\cI)$ are centered and independent, the symmetrization argument applies and, by definition of $r_M$,
\begin{align}\label{eq:multi-2}
 \nonumber \E\sup_{f\in B(f^*, \rho)}\sum_{k\in\mathcal K}\epsilon_k\frac{W_k(f)}{\gamma_k(f)}
&\leq \frac{4 K}{\epsilon r_M^2N} \E \sup_{f\in B(f^*, \rho) \cap B_2(\bayes, r_M)}\left|\sum_{i\in J} \eps_{i}\zeta_i (f-\bayes)(X_i)\right|\\
&\le \frac{4 K}{\epsilon N}\gamma_M|\cK|\frac NK=\frac{4\gamma_M}{\epsilon}|\cK|\enspace.
\end{align}
Now, let $\psi(t) = (2t-1)I(1/2\leq t \leq 1)+I(t\ge 1)$ for all $t\geq0$ and note that $\psi$ is $2$-Lipschitz, $\psi(0)=0$ and satisfies $I(t\ge 1)\le \psi(t)\le I(t\ge 1/2)$ for all $t\geq0$. Therefore, all $f\in B(f^*, \rho)$ satisfies
\begin{align*}
\sum_{k\in \cK}I&\left(|W_k(f)| < \gamma_k(f)\right)\\
&=|\cK|-\sum_{k\in \cK} I\left(\frac{|W_k(f)|}{\gamma_k(f)}\ge 1\right)\\
&\ge|\cK|-\sum_{k\in \cK}\psi\left(\frac{|W_k(f)|}{\gamma_k(f)}\right)\\
&=|\cK|-\sum_{k\in\cK}\bE \psi\left(\frac{|W_k(f)|}{\gamma_k(f)}\right)-\sum_{k\in \cK}\left[\psi\left(\frac{|W_k(f)|}{\gamma_k(f)}\right)-\bE \psi\left(\frac{|W_k(f)|}{\gamma_k(f)}\right)\right]\\
&\ge |\cK|-\sum_{k\in\cK}\bP\left(\frac{|W_k(f)|}{\gamma_k(f)}\ge \frac12\right)-\sum_{k\in \cK}\left[\psi\left(\frac{|W_k(f)|}{\gamma_k(f)}\right)-\bP\psi\left(\frac{|W_k(f)|}{\gamma_k(f)}\right)\right]\\
&\ge(1-\alpha)|\cK|-\sup_{f\in B(f^*, \rho)}\left|\sum_{k\in \cK}\left[\psi\left(\frac{|W_k(f)|}{\gamma_k(f)}\right)-\bE \psi\left(\frac{|W_k(f)|}{\gamma_k(f)}\right)\right]\right|
\end{align*}where we used \eqref{eq:multi_1} in the last inequality.
The bounded difference inequality ensures that, for all $x>0$, there exists an event $\Omega(x)$ satisfying $\bP(\Omega(x))\ge 1-\exp(-x^2|\cK|/2)$ on which
\begin{multline*}
\sup_{f\in B(f^*, \rho)} \left|\sum_{k\in \cK}\left[\psi\left(\frac{|W_k(f)|}{\gamma_k(f)}\right)-\bE\psi\left(\frac{|W_k(f)|}{\gamma_k(f)}\right)\right]\right|\\
\le \bE \sup_{f\in B(f^*, \rho)}\left|\sum_{k\in \cK}\left[\psi\left(\frac{|W_k(f)|}{\gamma_k(f)}\right)-\bE\psi\left(\frac{|W_k(f)|}{\gamma_k(f)}\right)\right]\right|+|\cK|x\enspace.
\end{multline*}
Furthermore, it follows from the symmetrization argument that
\begin{multline*}
\bE \sup_{f\in B(f^*, \rho)}\left|\sum_{k\in \cK}\left[\psi\left(\frac{|W_k(f)|}{\gamma_k(f)}\right)-\bE\psi\left(\frac{|W_k(f)|}{\gamma_k(f)}\right)\right]\right|\\
\le 2\bE\sup_{f\in B(f^*, \rho)}\left|\sum_{k\in \cK}\epsilon_k\psi\left(\frac{|W_k(f)|}{\gamma_k(f)}\right)\right| 
\end{multline*}
and, from the contraction principle and \eqref{eq:multi-2}, that
\[
\bE\sup_{f\in B(f^*, \rho)}\left|\sum_{k\in \cK}\epsilon_k\psi\left(\frac{|W_k(f)|}{\gamma_k(f)}\right)\right|\le 2\bE \sup_{f\in B(f^*, \rho)}\left|\sum_{k\in \cK}\epsilon_k\frac{|W_k(f)|}{\gamma_k(f)}\right| \le \frac{8\gamma_M}{\eps}|\cK|\enspace.
\] 
In conclusion, on $\Omega(x)$, for all $f\in B(f^*, \rho)$,
\begin{align*}
\sum_{k\in \cK}I\left(|W_k(f)|< \gamma_k(f)\right)&\ge \left(1-\alpha-x-8 \gamma_M/\eps\right)|\cK|\\
&\geq K \gamma \left(1-\alpha-x-8 \gamma_M/\eps\right)\geq (1-\eta)K\enspace.
\end{align*}
\end{proof}

\subsection{An isometry property of $\MOM{K}{\cdot}$ processes}
 Besides the controls of the quadratic and multiplier MOM processes presented in Lemmas~\ref{lem:UBQP} and~\ref{lem:proc_multiplicatif} respectively, the estimation error bounds for the MOM estimators rely on the following isometry property of the MOM processus $f\in F\to \MOM{K}{|f-f^*|}$.
 
\begin{Lemma}\label{lem:Isometry}[Isometry property of the $\MOM{K}{\cdot}$ process]
Grant Assumptions~\ref{ass:Mom2F} and~\ref{ass:small-ball}. Fix $\eta\in (0,1)$, $\rho>0$ and let $\alpha,\gamma_Q, \gamma,x$ denote absolute constants in $(0,1)$ such that $\gamma\left(1-\alpha-x- 4 \theta_r\gamma_Q/\alpha\right)\ge 1-\eta$. Let $K\in [K_o/(1-\gamma), N\alpha/(2\theta_0 \theta_r)^2]$. There exists an event $\Omega_{iso}(K, \rho)\subset\Omega_Q(K, \rho)$ such that $\bP(\Omega_{iso}(K, \rho))\ge1-2\exp\left(-\gamma x^2 K/2\right)$ and on the event  $\Omega_{iso}(K, \rho)$, for all $f\in B(f^*, \rho)$, 
\begin{equation*}
 Q_{1-\eta,K}{|f-\bayes|}\le  \theta_r \norm{f-f^*}_{L^2_P} + \frac{4\theta_r}{\alpha}\max\left(r_Q(\rho, \gamma_Q), \norm{f-\bayes}_{L^2_P}\right) 
\end{equation*} and if  $\norm{f-f^*}_{L^2_P}\geq r_Q(\rho,\gamma_Q)$ then $Q_{\eta,K}{|f-\bayes|}\ge (1/(4\theta_0))\norm{f-\bayes}_{L^2_P}$. 

In particular, for $\eta=1/2$, on the event  $\Omega_{iso}(K, \rho)$, for all $f\in B(f^*, \rho)$, if  $\norm{f-f^*}_{L^2_P}\geq r_Q(\rho,\gamma_Q)$, then 
\begin{equation}\label{eq:iso_MOMK}
\frac{1}{4\theta_0}\norm{f-\bayes}_{L^2_P} \leq  \MOM{K}{|f-f^*|} \leq \theta_r \left(1+\frac{4}{\alpha}\right)\norm{f-\bayes}_{L^2_P}.
\end{equation}
\end{Lemma}
\begin{proof}
It follows from Lemma~\ref{lem:UBQP} that on the event $\Omega_Q(K, \rho)$ for all $f\in B(f^*, \rho)$, if $\norm{f-f^*}_{L^2_P}\geq r_Q(\rho, \gamma_Q)$ then $Q_{\eta,K}{|f-\bayes|}\ge (1/(4\theta_0)\norm{f-\bayes}_{L^2_P}$. This yields the ``lower bound'' result in \eqref{eq:iso_MOMK}. 

For the upper bound of the isomorphic result, we essentially repeat the proof of Lemma~\ref{lem:proc_multiplicatif}. Let us just highlight the main differences. We will use the same notation as in the proof of Lemma~\ref{lem:proc_multiplicatif} except that for all $f\in F$, we define 
\begin{equation*}
 W_k(f) = (P_{B_k} - \overline{P}_{B_k})|f-f^*| \mbox{ and } \gamma_k(f) = \frac{4\theta_r}{\alpha}  \max\left(r_Q(\rho, \gamma_Q), \norm{f-f^*}_{L^2_P}\right).
 \end{equation*} 

 It follows from Chebyshev's inequality and Assumption~\ref{ass:Mom2F} that
 \begin{align*}
 \bP\left[2 |W_k(f)|\geq \gamma_k(f)\right]\leq \frac{4  \overline{P}_{B_k}|f-f^*|}{\gamma_k(f)}\leq \frac{4 \theta_r \norm{f-f^*}_{L^2_P}}{\gamma_k(f)}\leq \alpha.
 \end{align*}Moreover, by convexity of $F$, we have, for $r_Q:=r_Q(\rho, \gamma_Q)$, 
\begin{align*}
(\star) &: =\E \sup_{f\in B(f^*, \rho)} \sum_{k\in\cK} \eps_k \frac{W_k(f)}{\gamma_k(f)}\\
& \leq \frac{4\theta_r}{\alpha r_Q} \E \sup_{f\in B(f^*, \rho)\cap S_2(f^*, r_Q)} \left|\sum_{k\in\cK} \eps_k (P_{B_k} - \overline{P}_{B_k})|f-f^*|\right|
\end{align*}and then using a symmetrization argument, we obtain that
\begin{align*}
(\star) \leq \frac{4\theta_r K}{\alpha r_Q N} \E \sup_{f\in B(f^*, \rho)\cap S_2(f^*, r_Q)} \left|\sum_{i\in J} \eps_i (f-f^*)(X_i)\right|\leq \frac{4 \theta_r \gamma_Q |\cK| }{\alpha }.
\end{align*}Finally, using the same argument as in the proof of Lemma~\ref{lem:proc_multiplicatif}, for all $x>0$ there exists an event $\Omega(x)$ such that $\bP(\Omega(x))\geq 1-\exp(-x^2|\cK|/2)$, on which for all $f\in B(f^*, \rho)$, 
\begin{equation*}
\sum_{k\in\cK} I(|W_k(f)|\leq \gamma_k(f))\geq |\cK|(1-\alpha- x -4\theta_r \gamma_Q/\alpha)\geq (1-\eta)|\cK|.
\end{equation*}In particular, on the event $\Omega(x)$, for all $f\in B(f^*, \rho)$ there are more than $(1-\eta)K$ blocks $B_k$ for which, $P_{B_k}|f-f^*|\leq \overline{P}_{B_k}|f-f^*| + \gamma_k(f)$. Now, the result follows from Assumption~\ref{ass:Mom2F} since $\overline{P}_{B_k}|f-f^*|\leq \theta_r\norm{f-f^*}_{L^2_P}$.

 \end{proof}

\subsection{Conclusion to the proof of Theorem~\ref{theo:basic-combining-loss-and-reg}}\label{sub:EndOfProof}
The proof relies on the following proposition.
\begin{Proposition}\label{prop:Useful}
Grant conditions of Theorem~\ref{theo:basic-combining-loss-and-reg}. Let $\gamma_Q=1/(661\theta_0)$, $\gamma_M=\eps/168$ for some $\eps < 7/(662 \theta_0^2)$ and the regularization parameter be such that
\begin{equation*}
\frac{20\eps r^2(\rho_K)}{7 \rho_K}< \lambda< \frac{10 r^2(\rho_K)}{331\theta_0^2 \rho_K}.
\end{equation*}
 The event $\Omega_0(K) = \Omega_Q(K, \rho_K)\cap \Omega_M(K, \rho_K)$ is such that $\bP(\Omega_0(K))\ge 1-2\exp\left(-K/1008\right)$ and on $\Omega_0(K)$ for all $f\in F$ if  $\|f-\bayes\|_{L_P^2}\ge r(\rho_K)$ or $\|f-\bayes\|\ge \rho_K$ then
\begin{equation*}
\text{MOM}_K\left[\ell_f-\ell_{f^*}\right]+\lambda(\|f\|-\|\bayes\|)> 0\enspace. 
\end{equation*}
\end{Proposition}

\begin{proof}
 Using \eqref{prop:Cone}, \eqref{prop:Opposes} and \eqref{prop:Sum} together with the quadratic / multiplier decomposition of the excess quadratic loss yields that for all $f\in F$, 
\begin{align}\label{Obj0}
\notag\text{MOM}_K\left[\ell_f-\ell_{f^*}\right] &= \text{MOM}_K\left[(f-\bayes)^2-2\zeta(f-\bayes)\right]\\
&\ge Q_{1/4,K}[(f-\bayes)^2]-2Q_{3/4,K}[\zeta(f-\bayes)]\enspace.
\end{align}

Note that $\gamma(1-\alpha - x - 32\theta_0 \gamma_Q) \geq 1-\eta$ when one chooses
\begin{equation}\label{eq:choice_constants}
\eta = \frac{1}{4}, \gamma = \frac{7}{8},  \alpha = \frac{1}{21}, x = \frac{1}{21}, \gamma_Q = \frac{1}{661 \theta_0}, \gamma_M = \frac{\eps}{168} \mbox{ and } \eps \leq  \frac{1}{64 \theta_0^2}.
\end{equation}For this choice of constants,  Lemma~\ref{lem:UBQP} applies and for $\rho=\rho_K$ we get that there exists an event $\Omega_Q(K, \rho_K)$ with probability larger than $1-\exp(-K/1008)$ and on that event, for all $f\in B(\bayes,\rho_K)$, if $\norm{f-f^*}_{L^2_P}\geq r_Q(\rho_K,\gamma_Q)$ then
\begin{equation}\label{eq:quad_final_proof}
Q_{1/4,K}[(f-\bayes)^2]\ge \frac1{(4\theta_0)^2}\norm{f-\bayes}^2_{L^2_P}\enspace.
\end{equation}
Moreover, for the choice of parameters as in \eqref{eq:choice_constants}, we also have $\gamma(1-\alpha-x-8\gamma_M/\eps)\geq 1-\eta$, hence Lemma~\ref{lem:proc_multiplicatif} applies and for $\rho=\rho_K$ we get that there exists an  event $\Omega_M(K, \rho_K)$ with probability larger than $1-\exp(-K/1008)$ and on that event, for all $f\in B(\bayes,\rho_K)$ there are more than $3K/4$ blocks $B_k$ with $k\in \cK$ such that
\begin{equation*}
|2(P_{B_k}-\overline P_{B_k})\zeta(f-\bayes)|\le \eps\max\left(\frac{16 \theta_m^2}{\eps^2\alpha}\frac{K}{N}, r^2_M(\rho_K, \gamma_M), \norm{f-\bayes}^2_{L^2_P}\right)\enspace.
\end{equation*}
Combining the last result with Assumpion~\eqref{eq:robust_theo_basic}, it follows that on the event $\Omega_M(K, \rho_K)$, for all $f\in B(\bayes,\rho_K)$,
\begin{equation}\label{eq:multi_final_proof}
2Q_{3/4,K}[\zeta(f-\bayes)]\le 2\eps\max\left(\frac{16 \theta_m^2}{\eps^2\alpha}\frac{K}{N}, r^2_M(\rho_K, \gamma_M), \norm{f-\bayes}^2_{L^2_P}\right)\enspace.
\end{equation}
%Now, since $\mathcal K\subset K$, $Q_{3/4,K}[\zeta(f-\bayes)]\le Q_{3/4,\mathcal K}[\zeta(f-\bayes)]$, where $\mathcal K=\{k : B_k\cap \mathcal O=\emptyset\}$, and, $\forall z\in \R^N$, 
%\[
%Q_{3/4,\mathcal K}[z]=\inf\{x\in \R : \sum_{k\in\mathcal K}I(P_{B_k}z\le x)\ge \frac{3}4K\}\enspace.
%\]
%Then, by \eqref{eq:robust_theo_basic}, for any $i\in I$, $P_i\zeta(f-\bayes)\le c_3\beta^2(\|f-\bayes\|^2_{L^2_P}\vee r(\rho_K)^2)$, so 
%\begin{align}
%\notag Q_{3/4,\mathcal K}[\zeta(f-\bayes)]&\le Q_{3/4,\mathcal K}^{\overline P}[\zeta(f-\bayes)]+c_3\beta^2(\|f-\bayes\|_{L^2_P}\vee r(\rho_K))\\
% &\le Q_{3/4,\mathcal K}^{\overline P}[|\zeta(f-\bayes)|]+c_3\beta^2(\|f-\bayes\|^2_{L^2_P}\vee r(\rho_K)^2)\enspace.\label{eq:BoundQMP}
%\end{align}
%Since $c_1\theta_0^2\param{m}^2\frac{K}N\ge r^2_Q(\rho_K,\gamma_Q)$, where $c_1=\frac{1025*32}{\alpha}$, i

Let us now prove that on the event $\Omega_M(K, \rho_K)\cap \Omega_Q(K, \rho_K)$, one has  for all $f\in B(\bayes,\rho_K)$,
\begin{equation}\label{Obj3}
\text{MOM}_K\left[(f-\bayes)^2-2\zeta(f-\bayes)\right]\ge -2\eps r^2(\rho_K) \enspace.
\end{equation}
Assume that $\Omega_M(K, \rho_K)\cap \Omega_Q(K, \rho_K)$ holds and let $f\in B(\bayes,\rho_K)$. First assume that $\norm{f-f^*}_{L^2_P}\geq r^2(\rho_K)$. Then, it follows from \eqref{Obj0}, \eqref{eq:quad_final_proof} and \eqref{eq:multi_final_proof}, the choice of $\eps$ in \eqref{eq:choice_constants} and the definition of $\rho_K$ that
\begin{align}
\label{Obj2}&\text{MOM}_K\left[(f-\bayes)^2-2\zeta(f-\bayes)\right]\ge\pa{\frac1{(4\theta_0)^2}-2\epsilon}\norm{f-\bayes}^2_{L^2_P} \ge\frac{\norm{f-\bayes}^2_{L^2_P}}{32\theta_0^2}.
\end{align}Now, if $\norm{f-f^*}_{L^2_P}\leq r^2(\rho_K)$ then it follows from \eqref{Obj0}, \eqref{eq:multi_final_proof} and the definition of $\rho_K$ that 
\begin{equation*}
\text{MOM}_K\left[(f-\bayes)^2-2\zeta(f-\bayes)\right]\ge -2\eps r^2(\rho_K)
\end{equation*}and \eqref{Obj3} follows.

\paragraph{Conclusion of the proof when the regularization distance is small (i.e. $\|f-\bayes\|\leq \rho_K$) and the $L^2_P$-distance is large (i.e. $\norm{f-\bayes}_{L_P^2}\ge r(\rho_K)$)} Let $f\in F$ be such that $\|f-\bayes\|\le \rho_K$ and  $\norm{f-\bayes}_{L_P^2}\ge r(\rho_K)$.  It follows from the triangular inequality that $\|f\|-\|\bayes\|\ge -\|f-\bayes\|\ge -\rho_K$. Combining this together with \eqref{Obj2}, it follows that
\begin{equation*}
\text{MOM}_K\left[\ell_f-\ell_{f^*}\right]+\lambda(\|f\|-\|\bayes\|)\ge  \frac{\norm{f-\bayes}^2_{L^2_P}}{32\theta_0^2} - \lambda \rho_K\geq\frac{r^2(\rho_K)}{32\theta_0^2} - \lambda \rho_K>0
\end{equation*} when $\lambda < r^2(\rho_K)/(32\theta_0^2 \rho_K)$.

\paragraph{Conclusion of the proof when the regularization distance is large (i.e. $\|f-\bayes\|\geq \rho_K$): the homogeneity argument}
\begin{Lemma}\label{lem:Hom1}
For all $f\in F$ such that $\|f-\bayes\|\ge \rho_K$ 
\begin{equation*}
\norm{f}-\norm{\bayes}\geq\sup_{z^*\in\Gamma_{\bayes}(\rho_K)}z^*(f-\bayes)-\frac{\rho_K}{10}\enspace.  
\end{equation*}
\end{Lemma}
\begin{proof}
 For every $f^{**}\in F^* +(\rho_K/20)B$ and every $z^*\in(\partial\norm{\cdot})_{f^{**}}$, 
\begin{align*}
\norm{f}&-\norm{\bayes}\geq \norm{f} - \norm{f^{**}} -\norm{f^{**}-\bayes}\geq z^*(f-f^{**})-\frac{\rho_K}{20} \\
&= z^*(f-\bayes)-z^*(f^{**}-\bayes)-\frac{\rho_K}{20}\geq z^*(f-\bayes)-\frac{\rho_K}{10}\enspace.
\end{align*}
\end{proof}

\begin{Lemma}\label{lem:Hom2}
Assume that, for all $f\in F\cap S(\bayes,\rho_K)$, 
\begin{equation}\label{Obj4}
\MOM{K}{(f-\bayes)^2-2\zeta(f-\bayes)}+\lambda\sup_{z^*\in\Gamma_{\bayes}(\rho_K)}z^*(f-\bayes)> \lambda\frac{\rho_K}{10}\enspace.
\end{equation}
Then \eqref{Obj4} holds for all $f\in F$ such that $\norm{f-\bayes}\ge \rho_K$.
\end{Lemma}
\begin{proof}
%
%Assume that, for any $g\in S(\bayes,\rho_K)$, 
%\begin{equation}\label{Obj4}
%\MOM{K}{(g-\bayes)^2-2\zeta(g-\bayes)}+\lambda\sup_{z^*\in\Gamma_{\bayes}(\rho_K)}z^*(g-\bayes)\ge \lambda\frac{\rho_K}{10}\enspace.
%\end{equation}
Let $f\in F$ be such that $\norm{f-\bayes}\ge \rho_K$. Define $g=\bayes+\rho_K\frac{f-\bayes}{\|f-\bayes\|}$ and remark that $\norm{g-f^*}_{L^2_P} = \rho_K$ and that, by convexity of $F$, $g\in F$. It follows from \eqref{Obj4} that for $\kappa=\|f-\bayes\|/\rho_K\ge 1$, one has 
\begin{align*}
 &\MOM{K}{(f-\bayes)^2-2\zeta(f-\bayes)}+\lambda\sup_{z^*\in\Gamma_{\bayes}(\rho_K)}z^*(f-\bayes)\\
 &=\MOM{K}{\kappa^2(g-\bayes)^2-2\kappa\zeta(g-\bayes)}+\lambda\kappa\sup_{z^*\in\Gamma_{\bayes}(\rho_K)}z^*(g-\bayes)\\
 &\ge \kappa\left(\MOM{K}{(g-\bayes)^2-2\zeta(g-\bayes)}+\lambda\sup_{z^*\in\Gamma_{\bayes}(\rho_K)}z^*(g-\bayes)\right)\\
 &> \kappa \lambda\frac{\rho_K}{10}\geq \lambda\frac{\rho_K}{10}\enspace. 
\end{align*}  
\end{proof}

%Using the bound \eqref{eq:BoundQMP} shows finally that \eqref{Obj0} holds for any $f$ such that $\|f-\bayes\|\ge \rho_K$ if, for any $f\in S(\bayes,\rho_K)$,
%\begin{equation}\label{Obj5}
%Q_{1/4,K}[(f-\bayes)^2]-2Q_{3/4,\mathcal K}^{\overline{P}}[|\zeta(f-\bayes)|]+\lambda\sup_{z^*\in\Gamma_{\bayes}(\rho_K)}z^*(f-\bayes)\ge \lambda\frac{\rho_K}{10}+2\alpha_c(\|f-\bayes\|_{L^2_P}^2\vee r(\rho_K)^2)\enspace. 
%\end{equation}
Let $f\in F$ be such that $\norm{f-\bayes}\ge \rho_K$. By Lemma~\ref{lem:Hom1},
\begin{align*}
 &\MOM{K}{(f-\bayes)^2-2\zeta(f-\bayes)}+\lambda(\|f\|-\|\bayes\|)\\
 &\ge  \MOM{K}{(f-\bayes)^2-2\zeta(f-\bayes)}+\lambda\sup_{z^*\in\Gamma_{\bayes}(\rho_K)}z^*(f-\bayes)-\lambda\frac{\rho_K}{10}\enspace.
\end{align*}
Therefore, it will follow from Lemma~\ref{lem:Hom2} that 
\[
\MOM{K}{(f-\bayes)^2-2\zeta(f-\bayes)}+\lambda(\|f\|-\|\bayes\|)> 0
\]
if we can prove that for all $g\in F$ such that $\|g-\bayes\|=\rho_K$ one has
\begin{equation}\label{eq:TODO}
 \MOM{K}{(g-\bayes)^2-2\zeta(g-\bayes)}+\lambda\sup_{z^*\in\Gamma_{\bayes}(\rho_K)}z^*(g-\bayes)>\lambda\frac{\rho_K}{10}\enspace.
\end{equation}

Let us now prove that \eqref{eq:TODO} holds. Let $g\in F$ be such that $\|g-\bayes\|=\rho_K$. First assume that $\norm{g-f^*}_{L^2_P}\leq r(\rho_K)$ so that $g\in H_{\rho_K}$. By definition $\sup_{z^*\in\Gamma_{\bayes}(\rho_K)}z^*(g-\bayes)\ge \Delta(\rho_K)$ and, since $\rho_K\ge \rho^*$, $\rho_K$ satisfies the sparsity equation and thus, $\sup_{z^*\in\Gamma_{\bayes}(\rho_K)}z^*(g-\bayes)\ge 4\rho_K/5$. Therefore, thanks to \eqref{Obj3}, when $\lambda> 20\eps r^2(\rho_K)/(7 \rho_K)$, one has
\begin{align*}
\MOM{K}{(g-\bayes)^2  -2\zeta(g-\bayes)}&+\lambda\sup_{z^*\in\Gamma_{\bayes}(\rho_K)}z^*(g-\bayes)\\
 &\ge -2 \eps r^2(\rho_K)+\lambda\frac{4}{5}\rho_K> \lambda\frac{\rho_K}{10}\enspace.
\end{align*}
Finally assume that $\|g-\bayes\|_{L^2_P}\ge r(\rho_K)$. Since $\sup_{z^*\in\Gamma_{\bayes}(\rho_K)}z^*(f-\bayes)\ge -\|f-\bayes\|=-\rho_K$, it follows from \eqref{Obj2} that
\begin{align*}
  &\MOM{K}{(g-\bayes)^2-2\zeta(g-\bayes)}+\lambda\sup_{z^*\in\Gamma_{\bayes}(\rho_K)}z^*(g-\bayes)\\
&\ge\frac1{32\param{0}^2}\norm{g-\bayes}^2_{L^2_P}-\lambda \rho_K \geq\frac{r^2(\rho_K)}{32\param{0}^2}-\lambda \rho_K >  \lambda\frac{\rho_K}{10}
\end{align*}when $\lambda< 10 r^2(\rho_K)/(331\theta_0^2 \rho_K)$.
\end{proof}

\paragraph{End of the proof of Theorem~\ref{theo:basic-combining-loss-and-reg}}
On the event $\Omega_0(K)$ of Proposition~\ref{prop:Useful}, $\cB_{K,\lambda}(\bayes)$ is included in the ball $B(\bayes,\rho_K)$, therefore, by definition of $\est^{(1)}$ (cf. \eqref{GenRiskBound}),
\[
\norm{\est^{(1)}-\bayes}\leq C_{K,\lambda}^{(1)}(\bayes)\le \rho_K\enspace.
\]

Again, by Proposition~\ref{prop:Useful}, on the same event $\Omega_0(K)$, $\cB_{K,\lambda}(\bayes)\subset B(f^*, \rho_K)\cap B_2(\bayes, r(\rho_K))$, hence, on $\Omega_0(K)\cap\Omega_{iso}(K)$, where $\Omega_{iso}(K)$ is an event defined in Lemma~\ref{lem:Isometry}, for all $f\in \cB_{K,\lambda}(\bayes)$, 
\[
\MOM{K}{|f-\bayes|}\le 85\param{r}\norm{f-\bayes}_{L^2_P}\le 85\param{r} r(\rho_K)
\]where $\alpha=1/21$ according to \eqref{eq:choice_constants}. Therefore, $C_{K,\lambda}^{(2)}(\bayes)\le \rho_K$, which implies that $\norm{\est^{(2)}-\bayes}\le \rho_K$ (cf. \eqref{GenRiskBound}) and that $C_{K,\lambda}^{(2)}(\est^{(2)})\le \rho_K$ and therefore, by Lemma~\ref{lem:Isometry}, on $\Omega_0(K)\cap\Omega_{iso}(K)$, either $\norm{\est^{(2)}-\bayes}_{L^2_P}\leq r_Q(\rho_K, \gamma_K)$ and so $\norm{\est^{(2)}-\bayes}_{L^2_P}\leq 340\param{0}\param{r}r(\rho_K)$ or $\norm{\est^{(2)}-\bayes}_{L^2_P}\geq r_Q(\rho_K, \gamma_K)$ and so
\[
\norm{\est^{(2)}-\bayes}_{L^2_P}\le 4\param{0}\MOM{K}{|\est^{(2)}-\bayes|}\le 340\param{0}\param{r}r(\rho_K)\enspace.
\] 

\subsection{Conclusion to the proof of Theorem~\ref{theo:main_baby_Lepski}}
First, it follows from Theorem~\ref{theo:basic-combining-loss-and-reg} that for all $K \in[K_1, K_2]$, with probability at least $1-\cabs{0}\exp(-\cabs{1}K)$, for both $j=1, 2$, $\bayes\in \cap_{J=K}^{K_2} R^{(j)}_K$, so $\widehat K^{(j)}\le K$, which implies that both $\bayes$ and $\ESTI{\text{LE}}{j}$ belong to $B(\ESTI{K,\lambda}{j},\rho_K)$, therefore, $\norm{\bayes-\ESTI{\text{LE}}{j}}\le 2\rho_K$.
%The result in $L^2_P$-norm follows from the following Lemma.
%\begin{Lemma}\label{lem:Lepski}
% Assume that $(\ESTI{K}{s})_{K\in [K_1,K_2]}$ is a sequence of estimators of $\bayes$ such that, for all $K\in [K_1,K_2]$,
%\begin{equation}\label{ass:RiskBound}
% \bP\left(\norm{\ESTI{K}{s}-\bayes}\le \rho_K,\norm{\bayes-\ESTI{K}{s}}_{L_P^2}\le r_K\right)\ge 1-c_1e^{-c_2K}\enspace.
%\end{equation}
% 
%Grant the assumptions of Lemma~\ref{lem:Isometry} and let 
%\[
%\hat K_{c_3}=\inf\left\{K\in[K_1,K_2] : \cap_{J=K}^{K_2}\left\{f\in B(\ESTI{J}{s},\rho_J) : \MOM{J}{|f-\ESTI{J}{s}|}\le \cabs{3}\param{r}r_J\right\}\ne \emptyset\right\}\enspace,
%\]
%\[
%\ESTI{\text{LE},c_3}{s}\in \cap_{J=\hat K_{c_3}}^{K_2}\left\{f\in B(\ESTI{J}{s},\rho_J) : \MOM{J}{|f-\ESTI{J}{s}|}\le \cabs{3}\param{r}r_J\right\}\enspace.
%\]
%Then, there exists absolute constants $(\cabs{i})_{3\le i\le 6}$ such that
%\[
%\forall K\in [K_1,K_2],\; \bP\left(\norm{\ESTI{K}{s}-\bayes}\le 2\rho_K,\;\norm{\ESTI{\text{LE},\cabs{3}}{s}-\bayes}_{L^2_P}\le (\cabs{4}\param{r}\param{0}r_K)\vee r_Q(2\rho_K,\gamma_Q)\right)\ge 1-\cabs{5}e^{-\cabs{6}K}\enspace.
%\]
%\end{Lemma}

\noindent
%\begin{proof}
%We have seen that $\norm{\ESTI{K}{s}-\bayes}\le 2\rho_K$ on the event $\Omega(K)$. 

Second, for the $L^2_P$-estimation  error bound of $\ESTI{\text{LE}}{2}$, denote by $r_J=340\param{r}\param{0}r(\rho_J)$ the bound on the $L^2_P$ risk of the estimator $\hat f^{(2)}_J$ obtained in Theorem~\ref{theo:basic-combining-loss-and-reg}. Let $K\in [K_1,K_2]$. It follows from Lemma~\ref{lem:Isometry} for $\rho=2\rho_J, J\geq K$ that there exists absolute constants $\cabs{1},\cabs{2}$ and an event $\Omega_{iso}$ such that $\bP(\Omega_{iso})\geq 1-\cabs{1}\exp(-\cabs{2}K)$ and, on the event $\Omega_{iso}$, for all $J\ge K$, $\eta\in \{1/4,1/2,3/4\}$ and $f\in B(\bayes,2\rho_J)$,
\[
\text{if }\norm{f-\bayes}_{L^2_P}\ge r_Q(2\rho_J,\gamma_Q),\qquad Q_{\eta,J}(|f-\bayes|)
\begin{cases}
\ge  \frac{1}{4\param{0}}\norm{f-\bayes}_{L^2_P}\\
\le 85\param{r}\norm{f-\bayes}_{L^2_P}
\end{cases}
\enspace.
\]
Let $\Omega$ be the event defined as the following intersection:
\begin{equation*}
\Omega =  \bigcap_{J=K}^{K_2}\left\{\norm{\ESTI{J}{2}-\bayes}\le\rho_J \mbox{ and }\norm{\ESTI{J}{2}-\bayes}_{L^2_P}\le r_J\right\} \bigcap  \Omega(K) \bigcap    \Omega_{iso}\enspace.
\end{equation*} It follows from Theorem~\ref{theo:basic-combining-loss-and-reg} that $\bP(\Omega)\ge 1-c_3 \exp(-c_4 K)$. Moreover, on $\Omega$, for all $J\ge K$, 
\[
Q_{3/4,J}\pa{|\bayes-\ESTI{J}{2}|}
%\le 85\param{r}\norm{\bayes-\ESTI{J}{2}}_{L^2_P}
\le 85\param{r}r_J\enspace.
\]
So, in particular, $\bayes\in\cap_{J=K}^{K_2}\left\{f\in B(\ESTI{J}{2},\rho_J) : \MOM{J}{|f-\ESTI{J}{2}|}\le 85\param{r}r_J\right\}$. By definition of $\hat K^{(2)}$, this implies that $\hat K^{(2)}\le K$ on $\Omega$. Therefore, on $\Omega$, 
\[
\hat f_{LE}^{(2)} \in \cap_{J=K}^{K_2}\left\{f\in B(\bayes,2\rho_J) : \MOM{J}{|f-\ESTI{J}{2}|}\le 85\param{r}r_J\right\}\enspace.
\]
In particular, 
\[
\MOM{K}{|\hat f_{LE}^{(2)}-\ESTI{K}{2}|}\le 85\param{r}r_K\enspace.
\]
Now on $\Omega_{iso}$, one has for all $f\in B(\bayes,2\rho_K)$, if $\norm{f-f^*}_{L^2_P}\geq r_Q(2\rho_K,\gamma_Q)$ then
\[
Q_{1/4,J}[|f-\bayes|]\ge \frac{1}{4\param{0}}\norm{f-\bayes}_{L^2_P}\enspace.
\] 
Therefore on $\Omega_{iso}$, one has either $\norm{\hat f_{LE}^{(2)}-\bayes}_{L^2_P}\le r_Q(2\rho_K,\gamma_Q)$ or $\norm{\hat f_{LE}^{(2)}-\bayes}_{L^2_P}\ge r_Q(2\rho_K,\gamma_Q)$ and in the latter case, 
\begin{align*}
\norm{\hat f_{LE}^{(2)}-\bayes}_{L^2_P}&\le4\param{0}Q_{1/4,K}[|\hat f_{LE}^{(2)}-\bayes|]\\
&\le 4\param{0}\left(\MOM{K}{\big|\hat f_{LE}^{(2)}-\ESTI{K}{2}\big|}+Q_{3/4,K}(|\ESTI{K}{2}-\bayes|)\right)\\
&\le 680\param{0} \param{r} r_K\enspace.
\end{align*}
\endproof

%% If you have bibdatabase file and want bibtex to generate the
%% bibitems, please use
%%
%%  \bibliographystyle{elsarticle-harv} 
%%  \bibliography{<your bibdatabase>}

\bibliographystyle{elsarticle-harv}
\bibliography{biblio}

\def\cprime{$'$}
\begin{thebibliography}{46}
\expandafter\ifx\csname natexlab\endcsname\relax\def\natexlab#1{#1}\fi
\expandafter\ifx\csname url\endcsname\relax
  \def\url#1{\texttt{#1}}\fi
\expandafter\ifx\csname urlprefix\endcsname\relax\def\urlprefix{URL }\fi

\bibitem[{Alon et~al.(1999)Alon, Matias, and Szegedy}]{MR1688610}
Alon, N., Matias, Y., Szegedy, M., 1999. The space complexity of approximating
  the frequency moments. J. Comput. System Sci. 58~(1, part 2), 137--147,
  twenty-eighth Annual ACM Symposium on the Theory of Computing (Philadelphia,
  PA, 1996).
\newline\urlprefix\url{http://dx.doi.org/10.1006/jcss.1997.1545}

\bibitem[{Audibert and Catoni(2011)}]{MR2906886}
Audibert, J.-Y., Catoni, O., 2011. Robust linear least squares regression. Ann.
  Statist. 39~(5), 2766--2794.
\newline\urlprefix\url{http://dx.doi.org/10.1214/11-AOS918}

\bibitem[{Baraud(2011)}]{MR2834722}
Baraud, Y., 2011. Estimator selection with respect to {H}ellinger-type risks.
  Probab. Theory Related Fields 151~(1-2), 353--401.
\newline\urlprefix\url{http://dx.doi.org/10.1007/s00440-010-0302-y}

\bibitem[{Baraud and Birg{\'e}(2009)}]{MR2449129}
Baraud, Y., Birg{\'e}, L., 2009. Estimating the intensity of a random measure
  by histogram type estimators. Probab. Theory Related Fields 143~(1-2),
  239--284.
\newline\urlprefix\url{http://dx.doi.org/10.1007/s00440-007-0126-6}

\bibitem[{Baraud and Birg{\'e}(2016)}]{MR3565484}
Baraud, Y., Birg{\'e}, L., 2016. Rho-estimators for shape restricted density
  estimation. Stochastic Process. Appl. 126~(12), 3888--3912.
\newline\urlprefix\url{http://dx.doi.org/10.1016/j.spa.2016.04.013}

\bibitem[{Baraud et~al.(2017)Baraud, Birg\'e, and Sart}]{BaraudBirgeSart}
Baraud, Y., Birg\'e, L., Sart, M., 2017. A new method for estimation and model
  selection : $\rho$-estimation. Invent. Math. 207~(2), 435--517.

\bibitem[{Baraud et~al.(2014)Baraud, Giraud, and Huet}]{MR3224300}
Baraud, Y., Giraud, C., Huet, S., 2014. Estimator selection in the {G}aussian
  setting. Ann. Inst. Henri Poincar\'e Probab. Stat. 50~(3), 1092--1119.
\newline\urlprefix\url{http://dx.doi.org/10.1214/13-AIHP539}

\bibitem[{Bellec et~al.(2016)Bellec, Lecu{\'e}, and Tsybakov}]{BLT16}
Bellec, P., Lecu{\'e}, G., Tsybakov, A., 2016. Slope meets lasso: Improved
  oracle bounds and optimality. Tech. rep., CREST, CNRS, Universit{\'e} Paris
  Saclay.

\bibitem[{Birg{\'e}(2006)}]{MR2219712}
Birg{\'e}, L., 2006. Model selection via testing: an alternative to (penalized)
  maximum likelihood estimators. Ann. Inst. H. Poincar\'e Probab. Statist.
  42~(3), 273--325.
\newline\urlprefix\url{http://dx.doi.org/10.1016/j.anihpb.2005.04.004}

\bibitem[{Birg{\'e}(2013)}]{MR3186748}
Birg{\'e}, L., 2013. Robust tests for model selection. In: From probability to
  statistics and back: high-dimensional models and processes. Vol.~9 of Inst.
  Math. Stat. (IMS) Collect. Inst. Math. Statist., Beachwood, OH, pp. 47--64.
\newline\urlprefix\url{http://dx.doi.org/10.1214/12-IMSCOLL905}

\bibitem[{Bogdan et~al.(2015)Bogdan, van~den Berg, Sabatti, Su, and
  Cand{\`e}s}]{slope1}
Bogdan, M., van~den Berg, E., Sabatti, C., Su, W., Cand{\`e}s, E.~J., 2015.
  S{LOPE}---{A}daptive variable selection via convex optimization. Ann. Appl.
  Stat. 9~(3), 1103--1140.

\bibitem[{Boucheron et~al.(2013)Boucheron, Lugosi, and Massart}]{BouLugMass13}
Boucheron, S., Lugosi, G., Massart, P., 2013. Concentration Inequalities: A
  Nonasymptotic Theory of Independence. Oxford University Press, iSBN
  978-0-19-953525-5.

\bibitem[{Brownlees et~al.(2015)Brownlees, Joly, and Lugosi}]{MR3405602}
Brownlees, C., Joly, E., Lugosi, G., 2015. Empirical risk minimization for
  heavy-tailed losses. Ann. Statist. 43~(6), 2507--2536.
\newline\urlprefix\url{http://dx.doi.org/10.1214/15-AOS1350}

\bibitem[{B{\"u}hlmann and van~de Geer(2011)}]{MR2807761}
B{\"u}hlmann, P., van~de Geer, S., 2011. Statistics for high-dimensional data.
  Springer Series in Statistics. Springer, Heidelberg, methods, theory and
  applications.
\newline\urlprefix\url{http://dx.doi.org/10.1007/978-3-642-20192-9}

\bibitem[{Catoni(2012)}]{MR3052407}
Catoni, O., 2012. Challenging the empirical mean and empirical variance: a
  deviation study. Ann. Inst. Henri Poincar\'e Probab. Stat. 48~(4),
  1148--1185.
\newline\urlprefix\url{http://dx.doi.org/10.1214/11-AIHP454}

\bibitem[{Chichignoud and Lederer(2014)}]{MR3217454}
Chichignoud, M., Lederer, J., 2014. A robust, adaptive {M}-estimator for
  pointwise estimation in heteroscedastic regression. Bernoulli 20~(3),
  1560--1599.
\newline\urlprefix\url{http://dx.doi.org/10.3150/13-BEJ533}

\bibitem[{de~la Pe{\~n}a and Gin{\'e}(1999)}]{MR1666908}
de~la Pe{\~n}a, V.~H., Gin{\'e}, E., 1999. Decoupling. Probability and its
  Applications (New York). Springer-Verlag, New York.
\newline\urlprefix\url{http://dx.doi.org/10.1007/978-1-4612-0537-1}

\bibitem[{Devroye et~al.(2016)Devroye, Lerasle, Lugosi, and
  Oliveira}]{MR3576558}
Devroye, L., Lerasle, M., Lugosi, G., Oliveira, R.~I., 2016. Sub-{G}aussian
  mean estimators. Ann. Statist. 44~(6), 2695--2725.
\newline\urlprefix\url{http://dx.doi.org/10.1214/16-AOS1440}

\bibitem[{Fan et~al.(2017)Fan, Li, and Wang}]{FanLiWang2016}
Fan, J., Li, Q., Wang, Y., 2017. Estimation of high dimensional mean regression
  in the absence of symmetry and light tail assumptions. J. R. Stat. Soc. Ser.
  B. Stat. Methodol. 79~(1), 247--265.
\newline\urlprefix\url{http://dx.doi.org/10.1111/rssb.12166}

\bibitem[{Giraud(2015)}]{MR3307991}
Giraud, C., 2015. Introduction to high-dimensional statistics. Vol. 139 of
  Monographs on Statistics and Applied Probability. CRC Press, Boca Raton, FL.

\bibitem[{Huber(1964)}]{MR0161415}
Huber, P.~J., 1964. Robust estimation of a location parameter. Ann. Math.
  Statist. 35, 73--101.

\bibitem[{Huber and Ronchetti(2009)}]{HubRonch2009}
Huber, P.~J., Ronchetti, E.~M., 2009. Robust Statistics. Wiley.

\bibitem[{Jerrum et~al.(1986)Jerrum, Valiant, and Vazirani}]{MR855970}
Jerrum, M.~R., Valiant, L.~G., Vazirani, V.~V., 1986. Random generation of
  combinatorial structures from a uniform distribution. Theoret. Comput. Sci.
  43~(2-3), 169--188.
\newline\urlprefix\url{http://dx.doi.org/10.1016/0304-3975(86)90174-X}

\bibitem[{Koltchinskii and Mendelson(2015)}]{MR3431642}
Koltchinskii, V., Mendelson, S., 2015. Bounding the smallest singular value of
  a random matrix without concentration. Int. Math. Res. Not. IMRN~(23),
  12991--13008.
\newline\urlprefix\url{http://dx.doi.org/10.1093/imrn/rnv096}

\bibitem[{Le~Cam(1973)}]{MR0334381}
Le~Cam, L., 1973. Convergence of estimates under dimensionality restrictions.
  Ann. Statist. 1, 38--53.
\newline\urlprefix\url{http://links.jstor.org/sici?sici=0090-5364(197301)1:1<38:COEUDR>2.0.CO;2-V&origin=MSN}

\bibitem[{Le~Cam(1986)}]{MR856411}
Le~Cam, L., 1986. Asymptotic methods in statistical decision theory. Springer
  Series in Statistics. Springer-Verlag, New York.
\newline\urlprefix\url{http://dx.doi.org/10.1007/978-1-4612-4946-7}

\bibitem[{Lecu{\'e} and Mendelson(2013)}]{LM13}
Lecu{\'e}, G., Mendelson, S., 2013. Learning subgaussian classes: Upper and
  minimax bounds. Tech. rep., CNRS, Ecole polytechnique and Technion.

\bibitem[{Lecu{\'e} and Mendelson(2014)}]{LM_compressed}
Lecu{\'e}, G., Mendelson, S., 2014. Sparse recovery under weak moment
  assumptions. Tech. rep., CNRS, Ecole Polytechnique and Technion, to appear in
  Journal of the European Mathematical Society.

\bibitem[{Lecu{\'e} and Mendelson(2016{\natexlab{a}})}]{LM_reg_comp}
Lecu{\'e}, G., Mendelson, S., 2016{\natexlab{a}}. Regularization and the
  small-ball method i: sparse recovery. Tech. rep., CNRS, ENSAE and Technion,
  I.I.T.

\bibitem[{Lecu{\'e} and Mendelson(2016{\natexlab{b}})}]{LM_reg_comp_2}
Lecu{\'e}, G., Mendelson, S., 2016{\natexlab{b}}. Regularization and the
  small-ball method ii: complexity dependent error rates. Tech. rep., CNRS,
  ENSAE and Technion, I.I.T.

\bibitem[{Ledoux and Talagrand(1991)}]{LT:91}
Ledoux, M., Talagrand, M., 1991. Probability in {B}anach spaces. Vol.~23 of
  Ergebnisse der Mathematik und ihrer Grenzgebiete (3) [Results in Mathematics
  and Related Areas (3)]. Springer-Verlag, Berlin, isoperimetry and processes.

\bibitem[{Lepski(1991)}]{MR1147167}
Lepski, O.~V., 1991. Asymptotically minimax adaptive estimation. {I}. {U}pper
  bounds. {O}ptimally adaptive estimates. Teor. Veroyatnost. i Primenen.
  36~(4), 645--659.
\newline\urlprefix\url{http://dx.doi.org/10.1137/1136085}

\bibitem[{Lugosi and Mendelson(2017)}]{LugosiMendelson2016}
Lugosi, G., Mendelson, S., 2017. Risk minimization by median-of-means
  tournaments. Preprint available on ArXive:1608.00757.

\bibitem[{McDiarmid(1989)}]{MR1036755}
McDiarmid, C., 1989. On the method of bounded differences. In: Surveys in
  combinatorics, 1989 ({N}orwich, 1989). Vol. 141 of London Math. Soc. Lecture
  Note Ser. Cambridge Univ. Press, Cambridge, pp. 148--188.

\bibitem[{Mendelson(2014{\natexlab{a}})}]{Shahar-COLT}
Mendelson, S., 2014{\natexlab{a}}. Learning without concentration. In:
  Proceedings of the 27th annual conference on Learning Theory COLT14. pp. pp
  25--39.

\bibitem[{Mendelson(2014{\natexlab{b}})}]{MR3364699}
Mendelson, S., 2014{\natexlab{b}}. A remark on the diameter of random sections
  of convex bodies. In: Geometric aspects of functional analysis. Vol. 2116 of
  Lecture Notes in Math. Springer, Cham, pp. 395--404.

\bibitem[{Mendelson(2015{\natexlab{a}})}]{Shahar-ACM}
Mendelson, S., 2015{\natexlab{a}}. Learning without concentration. J. ACM
  62~(3), Art. 21, 25.
\newline\urlprefix\url{http://dx.doi.org/10.1145/2699439}

\bibitem[{Mendelson(2015{\natexlab{b}})}]{shahar_general_loss}
Mendelson, S., 2015{\natexlab{b}}. Learning without concentration for a general
  loss function. Tech. rep., Technion and ANU, Canberra.

\bibitem[{Mendelson(2016)}]{shahar_gafa_ln}
Mendelson, S., 2016. On multiplier processes under weak moment assumptions.
  Tech. rep., Technion.

\bibitem[{Nemirovsky and Yudin(1983)}]{MR702836}
Nemirovsky, A.~S., Yudin, D.~B., 1983. Problem complexity and method efficiency
  in optimization. A Wiley-Interscience Publication. John Wiley \& Sons, Inc.,
  New York, translated from the Russian and with a preface by E. R. Dawson,
  Wiley-Interscience Series in Discrete Mathematics.

\bibitem[{Rudelson and Vershynin(2014)}]{RV_small_ball}
Rudelson, M., Vershynin, R., 2014. Small ball probabilities for linear images
  of high dimensional distributions. Tech. rep., University of Michigan,
  international Mathematics Research Notices, to appear. [arXiv:1402.4492].

\bibitem[{Saba et~al.(2008)Saba, Hoffman, Hornbaker, Bhave, and
  Tabakoff}]{eQTL}
Saba, L., Hoffman, P.~L., Hornbaker, C., Bhave, S.~V., Tabakoff, B., 2008.
  Expression quantitative trait loci and the phenogen database 31~(3).

\bibitem[{Sart(2014)}]{MR3224298}
Sart, M., 2014. Estimation of the transition density of a {M}arkov chain. Ann.
  Inst. Henri Poincar\'e Probab. Stat. 50~(3), 1028--1068.
\newline\urlprefix\url{http://dx.doi.org/10.1214/13-AIHP551}

\bibitem[{Su and Cand{\`e}s(2015)}]{slope2}
Su, W., Cand{\`e}s, E.~J., 2015. Slope is adaptive to unknown sparsity and
  asymptotically minimax. Tech. rep., Stanford University, to appear in The
  Annals of Statistics.

\bibitem[{Vapnik(1998)}]{MR1641250}
Vapnik, V.~N., 1998. Statistical learning theory. Adaptive and Learning Systems
  for Signal Processing, Communications, and Control. John Wiley \& Sons, Inc.,
  New York, a Wiley-Interscience Publication.

\bibitem[{Vapnik and Chervonenkis(1974)}]{MR0474638}
Vapnik, V.~N., Chervonenkis, A.~Y., 1974. Teoriya raspoznavaniya obrazov.
  {S}tatisticheskie problemy obucheniya. Izdat. ``Nauka'', Moscow.

\end{thebibliography}

%% else use the following coding to input the bibitems directly in the
%% TeX file.

\end{document}